%% file: article.tex
\documentclass[11pt, a4paper]{amsart}

\usepackage{microtype}
\usepackage{times}
\AtBeginDocument{
  \DeclareSymbolFont{AMSb}{U}{msb}{m}{n}
  \DeclareSymbolFontAlphabet{\mathbb}{AMSb}
}

\usepackage[dvipsnames]{xcolor}
\usepackage{graphicx}
\usepackage{tikz}
\usetikzlibrary{positioning, shapes.geometric, decorations.pathmorphing}

\usepackage{caption}
\usepackage{subcaption}

\usepackage{svg}
\svgsetup{
  inkscapeexe=inkscape,
  inkscapeopt=-D,
  inkscapeformat=pdf,
  inkscapearea=page,
  inkscapelatex=true
}
\graphicspath{ {img/} {svg-inkscape/} }

\usepackage[a4paper, pdftex, left=2cm, top=2cm, right=2cm, bottom=2cm]{geometry}
\usepackage[final]{pdfpages}
\usepackage{changepage}
\usepackage[foot]{amsaddr}

\usepackage{url}
\usepackage{hyperref}
\hypersetup{
    breaklinks=true,
    bookmarksopen=true,
    pdftitle={A convergence framework for energy minimisation of linear self-adjoint elliptic PDEs in nonlinear approximation spaces}, % Change here
    pdfauthor={Alexandre Magueresse, Santiago Badia}, % Change here
    pdfsubject={A convergence framework for energy minimisation of linear self-adjoint elliptic PDEs in nonlinear approximation spaces}, % Change here
    pdfkeywords={Energy minimisation, Nonlinear approximation, Nonlinear Céa's lemma}, % Change here
    colorlinks=true,
    linkcolor=blue,
    citecolor=blue,
    filecolor=blue,
    urlcolor=blue
}
\newcommand{\fig}[1]{Fig.~\ref{fig:#1}}

\newcommand{\alg}[1]{\hyperref[alg:#1]{Algorithm~\ref*{alg:#1}}}
\newcommand{\assum}[1]{\hyperref[assumption:#1]{Assumption~\ref*{assumption:#1}}}
\newcommand{\subassum}[2]{\hyperref[assumption:#1]{Assumption~\ref*{assumption:#1}.#2}}
\newcommand{\lem}[1]{\hyperref[lem:#1]{Lemma~\ref*{lem:#1}}}
\newcommand{\corol}[1]{\hyperref[corol:#1]{Corollary~\ref*{corol:#1}}}
\newcommand{\prop}[1]{\hyperref[prop:#1]{Proposition~\ref*{prop:#1}}}
\newcommand{\thm}[1]{\hyperref[thm:#1]{Theorem~\ref*{thm:#1}}}
\newcommand{\sect}[1]{\hyperref[sect:#1]{Section~\ref*{sect:#1}}}
\newcommand{\subsect}[1]{\hyperref[subsect:#1]{Subsection~\ref*{subsect:#1}}}
\newcommand{\app}[1]{\hyperref[app:#1]{Appendix~\ref*{app:#1}}}

\usepackage{csquotes}
\usepackage[backend=biber, defernumbers=false, maxbibnames=5, style=numeric-comp, sorting=none, isbn=false, bibencoding=utf8, safeinputenc, url=false, doi=true, giveninits=true]{biblatex}

\usepackage{float}
\usepackage{framed}
\usepackage{booktabs}
\usepackage{multirow}
\usepackage[inline]{enumitem}
\usepackage{algorithm}
\usepackage{algpseudocode}

\usepackage{amsthm}
\newtheorem*{remark}{Remark}
\newtheorem*{example}{Example}
\newtheorem{assumption}{Assumption}
\newtheorem{lemma}{Lemma}
\newtheorem{corollary}{Corollary}

\newtheorem{definition}{Definition}

\usepackage{amsmath, amsfonts, amssymb, amscd}
\usepackage{mathtools}

\newcommand{\mb}[1]{\boldsymbol{#1}}
\newcommand{\mc}[1]{\mathcal{#1}}
\newcommand{\md}[1]{\mathbb{#1}}

\let\epsilon\varepsilon
\let\phi\varphi

\newcommand{\NN}{\mathbb{N}}
\newcommand{\RR}{\mathbb{R}}

\newcommand{\oo}[2]{\left(#1, #2\right)}
\newcommand{\oc}[2]{\left(#1, #2\right]}
\newcommand{\cc}[2]{\left[#1, #2\right]}

\newcommand{\rg}[2]{\left\{#1 : #2\right\}}

\newcommand{\D}{\mathrm{D}}

\newcommand{\DX}{\nabla_{\md{X}}}
\newcommand{\DW}{\nabla_{\md{W}}}

\newcommand{\abs}[1]{{\left\vert #1 \right\vert}}
\newcommand{\norm}[1]{{\vert\kern-0.25ex\vert #1 \vert\kern-0.25ex\vert}}
\newcommand{\inner}[2]{{\left\langle #1, #2 \right\rangle}}

\newcommand{\dl}{n_{\mathrm{L}}}
\newcommand{\dnl}{n_{\mathrm{NL}}}
\newcommand{\best}[1]{#1^{\mathrm{best}}}
\newcommand{\red}[1]{#1^{\mathrm{red}}}
\newcommand{\proxmap}[3]{\mathrm{Prox}_{\psi}(#1, #2, #3)}
\newcommand{\gradmap}[3]{G_{\psi}(#1, #2, #3)}

\newcommand{\dd}[1]{\ \mathrm{d} #1}

\usepackage{acronym}
\acrodef{pde}[PDE]{partial differential equation}
\acrodef{fem}[FEM]{finite element method}
\acrodef{pinn}[PINN]{physics-informed neural network}
\acrodef{dof}[DOF]{degree of freedom}
\acrodefplural{dof}{degrees of freedom}
\acrodef{cg}[CG]{conjugate gradient}
\acrodef{pl}[PL]{Polyak--\L ojasiewicz}

\title[A convergence framework for energy minimisation in nonlinear approximation spaces]{A convergence framework for energy minimisation of linear self-adjoint elliptic PDEs in nonlinear approximation spaces}
\date{\today}

\author[A. Magueresse]{Alexandre Magueresse$^{\dagger\ast}$}
\address{$^\dagger$School of mathematics\\Monash university\\Clayton\\Victoria 3800\\Australia}
\email{alexandre.magueresse@monash.edu}
\author[S. Badia]{Santiago Badia$^{\dagger}$}
\email{santiago.badia@monash.edu}
\thanks{$^\ast$Corresponding author}

\begin{document}

\begin{abstract}
  \input{0.abstract}
\end{abstract}

\maketitle

\section{Introduction}
\label{sect:introduction}
\input{1.introduction}

\section{Abstract setting}
\label{sect:setting}
\input{2.setting}

\section{Algorithm}
\label{sect:algorithm}
\input{3.algorithm}

\section{Convergence analysis}
\label{sect:convergence}
\input{4.convergence}

\section{Conclusion}
\label{sect:conclusion}
\input{5.conclusion}

\section*{Acknowledgements}
\label{sect:acknowledgements}
\input{6.acknowledgements}

\appendix

\section{Proofs}
\label{sect:proofs}
\input{7.proofs}

\printbibliography

\end{document}

%% file: 0.abstract.tex
Recent years have seen the emergence of nonlinear methods for solving partial differential equations (PDEs), such as physics-informed neural networks (PINNs). While these approaches often perform well in practice, their theoretical analysis remains limited, especially regarding convergence guarantees. This work develops a general optimisation framework for energy minimisation problems arising from linear self-adjoint elliptic PDEs, formulated over nonlinear but analytically tractable approximation spaces. The framework accommodates a natural split between linear and nonlinear parameters and supports hybrid optimisation strategies: linear variables are updated via linear solves or steepest descent, while nonlinear variables are handled using constrained projected descent.
We establish both local and global convergence of the resulting algorithm under modular structural assumptions on the discrete energy functional, including differentiability, boundedness, regularity, and directional convexity. These assumptions are stated in an abstract form, allowing the framework to apply to a broad class of nonlinear approximation manifolds. In a companion paper [Magueresse, Badia (2025, arXiv:2508.17705)], we introduce a concrete instance of such a space based on overlapping free-knot tensor-product B-splines, which satisfies the required assumptions and enables geometrically adaptive solvers with rigorous convergence guarantees.

%% file: 1.introduction.tex
Recent years have seen a growing interest in solving variational problems and \acfp{pde} using nonlinear approximation spaces, where tunable parameters control the basis functions, allowing the discretisation to adapt dynamically to the solution. Unlike classical linear methods, such as the \ac{fem}, which relies on fixed basis functions, nonlinear approximation offers the potential to capture sharp gradients, layers, and singularities with significantly fewer degrees of freedom by concentrating resolution where needed. However, this flexibility comes at the cost of increased algorithmic and analytical complexity. The optimisation problems that arise are often non-convex, the approximation spaces may lack linear structure, and enforcing constraints, such as boundary conditions, mesh regularity, or boundedness, becomes nontrivial.

One of the most influential developments in this context is the rise of \acp{pinn} \cite{raissi2019physics}. By representing \ac{pde} solutions as neural networks trained to minimise residual-based losses, \acp{pinn} bypass mesh generation and operate on highly expressive, nonlinear function spaces. The Deep Ritz Method \cite{weinan2018deep} exemplifies another branch of nonlinear approximation for variational problems, where neural networks parameterise trial functions minimising energy functionals directly. This framework has successfully tackled problems in diverse fields, yet it also exposes some of the central challenges of nonlinear approximation: expensive numerical integration, difficulties in enforcing boundary conditions, sensitivity to hyperparameters, and highly non-convex optimisation landscapes. Despite growing empirical evidence, rigorous convergence analyses for \acp{pinn} remain scarce and typically apply only in asymptotic or simplified regimes \cite{de2024numerical, beck2022full}. This gap highlights a broader lack of theoretical foundations for nonlinear approximation methods in variational settings.

A related line of work explores free-knot B-splines, which introduce nonlinearity through adaptive knot placement while retaining the structure and locality of classical spline bases. Originating in one-dimensional data fitting \cite{jupp1978approximation, beliakov2004least, kovacs2019nonlinear}, these methods have been extended to higher dimensions via tensor-product constructions \cite{schutze2003bivariate, deng2004optimizing, zhang2016b}, often guided by heuristic adaptivity or constrained optimisation to maintain mesh quality. The approximation properties of linear free-knot splines are well understood through their equivalence with ReLU neural networks \cite{he2020relu, opschoor2020deep}, which have been extensively analysed in terms of expressive power and optimal approximation rates \cite{yarotsky2017error, petersen2018optimal, daubechies2022nonlinear}. However, the practical feasibility of achieving these rates---namely, how to compute or approximate the best representations numerically---remains poorly understood. Moreover, the use of free-knot spline spaces for the discretisation of \acp{pde} has received little attention and remains an open area of study.

In contrast, classical adaptive methods, such as $h$-, $p$-, and $r$-adaptivity, have long provided reliable strategies for controlling approximation error and distributing computational resources efficiently. These methods refine, enrich, or reposition discretisation elements based on error estimators, while maintaining robust mathematical properties like stability and convergence \cite{ainsworth1997posteriori, babuvska1994p}. Among them, $r$-adaptivity stands out for its conceptual proximity to nonlinear approximation: by relocating mesh nodes according to a parametric transformation \cite{huang2010adaptive, budd2009adaptivity}, it introduces a nonlinear dependence on discretisation parameters, echoing modern approaches based on parametric basis functions.

The convergence of first-order methods in unconstrained settings is by now classical. For smooth, non-convex problems, it is well-known that gradient descent converges to a quasi-stationary point in $O(\epsilon^{-2})$ iterates, where $\epsilon$ is the target gradient norm. Global convergence guarantees can be obtained under some kind of convexity or gradient growth assumption, for example the \ac{pl} inequality guarantee global convergence to critical points and even linear rates, without requiring convexity \cite{karimi2016linear}. Recent work has advanced the understanding of gradient descent dynamics for structured non-convex functionals, with a focus on characterising attraction regions under weak regularity conditions \cite{traonmilin2023basins}. Their insights on basin geometry and convergence pathways directly inspire key aspects of our analysis in the nonlinear approximation setting.

In constrained optimisation, projected gradient descent and its extensions, such as mirror descent, are standard tools for handling inequality, geometric, or manifold constraints. However, most theoretical results for these methods assume convexity of the objective or the feasible set, or rely on strong regularity conditions \cite{bubeck2015convex}. For structured non-convex problems, particularly those arising from variational formulations of \acp{pde}, such assumptions may fail, and convergence guarantees are much less mature. Designing and analysing algorithms that can robustly handle constraints in non-convex, parameter-dependent settings remains a central open challenge.

\subsection*{Contributions}

We develop a general analytical framework for variational problems posed over nonlinear approximation spaces, aiming to address these theoretical gaps. Our setting is abstract and encompasses a wide class of parameter-dependent spaces, independent of any specific choice of basis functions or representation. Under minimal structural assumptions---uniform coercivity, differentiability, boundedness, and a directional convexity condition on the discrete energy functional---we establish both local and global convergence results for alternating minimisation schemes coupling linear and nonlinear parameters.

We propose a general two-step optimisation strategy: the linear parameters must satisfy a sufficient energy decrease condition, which, for instance, can be realised via an exact linear solve, an inexact \ac{cg} step, or a simple gradient update. The idea of eliminating linear parameters to simplify nonlinear optimisation problems is classical, tracing back to \cite{lawton1971elimination} in the context of least-squares curve fitting. To take constraints into account, the nonlinear parameters are updated via mirror descent, a generalisation of projected gradient descent, though our analysis can accommodate broader update rules. This leads to a nonlinear analogue of Céa's lemma, quantifying the gap between the iterates and a best approximation up to an optimisation error. Our results highlight the importance of structural properties in establishing convergence guarantees, an aspect often lacking in the analysis of \acp{pinn} and related methods. Beyond its theoretical interest, our framework provides algorithmic insights that can inform the design of robust numerical methods for constrained variational problems on nonlinear approximation spaces.

The theoretical framework we develop in this work has been largely motivated by our companion paper \cite{companion}, where we apply it to tensor-product free-knot B-spline spaces. In that context, the interplay between the linear spline coefficients and the nonlinear knot positions naturally leads to alternating minimisation algorithms of the type studied here. Several modelling choices and algorithmic observations made in the companion paper have directly influenced the assumptions we adopt in this work. In particular, the geometric constraints on knot placement and the lack of global convexity prompted us to seek convergence guarantees under weaker conditions, such as directional convexity and boundedness of derivatives, rather than relying on stronger, but less applicable, global regularity or convexity assumptions. A key example is the differentiability of the discrete energy functional: while differentiability of the basis functions in the underlying Hilbert space is not strictly necessary, it simplifies the analysis and provides a practical guideline. We strike a balance by assuming differentiability of the energy while also giving a sufficient condition that is often easy to verify in applications. This analytical work thus provides a rigorous foundation for the algorithms explored in the companion study while extending beyond it to a more general class of nonlinear approximation spaces.

%% file: 2.setting.tex
\subsection{Continuous problem}

We consider a self-adjoint elliptic \ac{pde} posed in weak (variational) form, via the minimisation of an associated energy functional. Let $a: U \times U \to \RR$ be a symmetric, coercive, and continuous bilinear form, and let $\ell: U \to \RR$ be a continuous linear form, where $U$ is a suitable Hilbert space. We seek $u^{\star} \in U$ such that
$$a(u^{\star}, v) = \ell(v), \qquad \forall v \in U.$$
The well-posedness of this problem relies on the coercivity and continuity of the forms: for all $u, v \in U$,
$$a(u, u) \geq \alpha \norm{u}_{U}^{2}, \qquad \abs{a(u, v)} \leq \norm{a}_{U \times U} \norm{u}_{U} \norm{v}_{U}, \qquad \abs{\ell(v)} \leq \norm{\ell}_{U} \norm{v}_{U},$$
for some coercivity constant $\alpha > 0$. The continuity constants $0 \leq \norm{a}_{U \times U}, \norm{\ell}_{U} < \infty$ coincide with the operator norm of $a$ and $\ell$, defined respectively as
$$\norm{a}_{U \times U} \doteq \sup_{u \in U} \sup_{v \in U} \frac{\abs{a(u, v)}}{\norm{u}_{U} \norm{v}_{U}}, \qquad \norm{\ell}_{U} \doteq \sup_{v \in U} \frac{\abs{\ell(v)}}{\norm{v}_{U}}.$$
The Lax--Milgram lemma then ensures the existence and uniqueness of the solution. Moreover, $u^{\star}$ is the minimiser of the \emph{energy functional} $\mc{J}: U \to \RR$ defined by
$$\mc{J}(u) \doteq \frac{1}{2} a(u, u) - \ell(u).$$
To quantify approximation quality, we introduce the energy norm $\norm{u}_{a}^{2} \doteq a(u, u)$, which reflects the natural topology induced by the variational problem. Although the exact energy $\mc{J}(u^{\star})$ is typically unknown, the identity
\begin{equation}
    \label{eq:energy-gap}
    \mc{J}(u) - \mc{J}(u^{\star}) = \frac{1}{2} \norm{u - u^{\star}}_{a}^{2}
\end{equation}
for all $u \in U$ shows that, up to an additive constant, the energy functional provides a direct measure of proximity to the global minimiser in the energy norm. This observation motivates its use not only as an optimisation target but also as a surrogate error indicator, which is particularly relevant in adaptive settings where refinement is guided by energetic considerations.

While the proposed methodology applies broadly to such variational problems, we illustrate it using two model cases: a function approximation problem and a diffusion-reaction (Poisson-type) equation. We will use these examples to highlight regularity and differentiability aspects of the nonlinear approximation spaces.

\subsubsection{Notations}

Let $\Omega \subset \RR^{d}$ be a bounded domain with Lipschitz boundary $\Gamma = \partial \Omega$. Define the Sobolev spaces $H^{0}(\Omega) = L^{2}(\Omega)$, consisting of square-integrable functions on $\Omega$; $H^{1}(\Omega)$, consisting of functions in $L^{2}(\Omega)$ with square-integrable weak derivatives; and $H^{1}_{0}(\Omega) = \{v \in H^{1}(\Omega) : v|_{\Gamma} = 0\}$, consisting of functions in $H^{1}(\Omega)$ with zero trace on $\Gamma$. Let also $H^{-1}(\Omega)$ denote the (topological) dual of $H^{1}_{0}(\Omega)$.

\subsubsection{Function approximation problem}

Given $f \in L^{2}(\Omega)$, the function approximation problem consists of finding $u \in L^{2}(\Omega)$ such that $u = f$ in $\Omega$. Its variational formulation corresponds to
$$a(u, v) \doteq \int_{\Omega} u v \dd{\Omega}, \qquad \ell(v) \doteq \int_{\Omega} f v \dd{\Omega}.$$

\subsubsection{Diffusion-reaction problem}

In strong form, the diffusion-reaction problem is: find $u \in H^{1}(\Omega)$ such that
$$-\nabla \cdot (\mb{K} \nabla u) + \sigma u = f \text{ in } \Omega, \qquad u = g \text{ on } \Gamma,$$
where $\mb{K} \in L^{\infty}(\Omega, \RR^{d \times d})$ is a symmetric, positive-definite diffusivity matrix, $\sigma \in L^{\infty}(\Omega)$ with $\sigma \geq 0$ is the reaction coefficient, $f \in H^{-1}(\Omega)$ is a source term, and $g \in H^{1/2}(\Gamma)$ is the boundary condition. For simplicity, we consider only Dirichlet boundary conditions, but Neumann or Robin boundary conditions could also be treated. Let $\bar{u} \in H^{1}(\Omega)$ be a lifting of the Dirichlet boundary conditions such that $\bar{u} = g$ on $\Gamma$. The weak form of this problem corresponds to
$$a(u, v) \doteq \int_{\Omega} (\mb{K} \nabla u \cdot \nabla v + \sigma u v) \dd{\Omega}, \qquad \ell(v) \doteq \int_{\Omega} (f + \nabla \cdot (\mb{K} \nabla \bar{u}) - \sigma \bar{u}) v \dd{\Omega},$$
where the integral $\int_{\Omega} f v \dd{\Omega}$ is understood as a duality pairing between $H^{-1}(\Omega)$ and $H^{1}_{0}(\Omega)$.

\subsection{Energy minimisation in nonlinear approximation spaces}

We aim to approximate the exact solution $u^{\star}$ in a finite-dimensional space $V \subset U$ by minimising the energy $\mc{J}$ over $V$. Unlike traditional methods where $V$ is a fixed linear subspace, we assume $V$ is a smoothly parameterised manifold of functions, defined as the image of a \emph{realisation map} $\mc{R} : \Theta \to U$ from a finite-dimensional \emph{parameter space} $\Theta$. The corresponding discrete problem becomes the minimisation of the \emph{discrete energy functional}
$$\mc{K} \doteq \mc{J} \circ \mc{R}: \Theta \to \RR.$$

Not all approximation spaces are suitable for energy minimisation. To ensure that the problem is well-posed, amenable to gradient-based methods, and numerically tractable, we impose structural assumptions on the parameter space and the realisation map.
\begin{description}
    \item[Existence of minimisers] Any continuous, coercive, and lower-bounded function defined on a compact set admits global minimisers \cite[Chapter 2]{ekeland1999convex}. Since $\mc{J}$ satisfies these properties on $U$, they are automatically transferred to the restriction of $\mc{J}$ to $V$, and it suffices to ensure that $V$ is closed in $U$. Assuming that the realisation map is continuous, $V$ is closed if $\Theta$ is compact.
    \item[Approximability] The nonlinear space $V$ must have strong approximation properties in $U$. Ideally, the decay rate of the approximation error with respect to the number of parameters should match or exceed that of classical schemes.
    \item[Differentiability and regularity] Gradient-based optimisation requires that the discrete energy $\mc{K}$ be differentiable, with a uniformly continuous gradient on $\Theta$, ensuring convergence of standard optimisation algorithms \cite[Chapter 1]{nesterov2013introductory}.
    \item[Well-conditioned parameterisation] As pointed out in \cite{petersen2021topological}, another key factor in ensuring reliable convergence is the comparability of the function norm in $U$ and the parameter norm in $\Theta$. This requires the realisation map $\mc{R}$ to be a uniformly continuous homeomorphism with a uniformly continuous inverse, thereby ensuring a well-conditioned parameterisation.
    \item[Computational feasibility] Efficient and reliable evaluation of both $\mc{K}$ and its gradient is crucial, especially when $\mc{J}$ contains integrals approximated by quadrature. Controlling the accuracy of these approximations is essential to prevent artificial minima or spurious oscillations in the optimised solution \cite{mishra2023estimates}.
\end{description}

We now examine in more detail the differentiability of the energy functional and the approximation properties of nonlinear spaces.

\subsubsection{Differentiability of the discrete energy}
\label{subsect:differentiability}

The differentiability of the energy functional with respect to the nonlinear parameters is a delicate matter that cannot be answered by a naive application of the chain rule. While the realisation map $\mc{R}$ is often weakly differentiable, it may fail to be differentiable into the function space $U$ on which the energy $\mc{J}$ is defined. As a result, the variations generated by the differential $\D \mc{R}$ along parameter directions may not belong to the tangent space of $U$, making the composition $\D \mc{J} \circ \D \mc{R}$ ill-defined.

The core difficulty stems from the fact that differentiation with respect to parameters does not, in general, commute with spatial integration. This mismatch often manifests as a loss of regularity: the derivative of $\mc{R}$ may produce variations of lower smoothness than required for admissible test directions in the domain of $\D \mc{J}$. Geometrically, this reflects a failure of the tangent space to the approximation manifold $V$ to embed within the tangent space of the ambient function space $U$. The following example illustrates that the inclusion $\D \mc{R} \subset U$ is not necessary for the differentiability, or even continuous differentiability, of the discrete energy.

\begin{example}
    Let $\Omega \subset \RR$ be a bounded interval. Given $I \subset \RR$, let $\chi_{I}: \RR \to \{0, 1\}$ denote the indicator function of $I$. Consider the realisation
    $$\mc{R}(\theta) = w_{1} \chi_{\oo{a}{b}} + w_{2} \chi_{\oo{b}{c}},$$
    for $\theta = (w_{1}, w_{2}, a, b, c) \in \Theta$, where $\Theta \subset \RR^{5}$ is the subset enforcing the constraints $a \leq b \leq c \in \Omega$. Since $\Omega$ is bounded, it is easy to see that $\mc{R}(\theta) \in L^{2}(\RR)$ for all $\theta \in \Theta$. Still, $\mc{R}$ is not differentiable in $L^{2}(\RR)$, but only in $H^{-1}(\RR)$, with
    $$\D \mc{R}(\theta) = \chi_{\oo{a}{b}} \dd{w_{1}} + \chi_{\oo{b}{c}} \dd{w_{2}} - w_{1} \delta_{a} \dd{a} + (w_{1} - w_{2}) \delta_{b} \dd{b} + w_{2} \delta_{c} \dd{c}.$$
    Here $\delta_{z} \in H^{-1}(\RR)$ denotes the Dirac delta centred at $z \in \RR$. For the function approximation problem, we compute
    $$a(\mc{R}(\theta), \mc{R}(\theta)) = w_{1}^{2} (b - a) + w_{2}^{2} (c - b), \qquad \ell(\mc{R}(\theta)) = w_{1} \int_{\oo{a}{b} \cap \Omega} f \dd{\Omega} + w_{2} \int_{\oo{b}{c} \cap \Omega} f \dd{\Omega}.$$
    In particular, $a(\mc{R}(\theta), \mc{R}(\theta))$ is infinitely differentiable in $\theta$, and $\ell(\mc{R}(\theta))$ has the same regularity as $f$ in $\Omega$.
\end{example}

Whether or not this type of regularity issue arises depends on the properties of both the realisation map and the structure of the approximation space. As such, establishing the differentiability of the discrete energy must be addressed on a case-by-case basis.

\subsubsection{Nonlinear approximation rates}

The effectiveness of a nonlinear space $V$ hinges on its approximation power. Ideally, when $V$ belongs to a nested sequence of approximation spaces, the associated approximation error should decrease at least as rapidly as in standard methods. In standard discretisations such as finite elements, finite volumes or finite differences, convergence rates are expressed in terms of mesh size assuming quasi-uniformity. However, more general approximation spaces may lack a natural length scale, making it more appropriate to measure convergence relative to the number of parameters. In this context, classical convergence rates must be interpreted carefully, as they depend on the spatial dimension.

\subsection{Separation of linear and nonlinear parameters}

In many cases, it is natural for the realisation map to depend linearly on some parameters and nonlinearly on others. This separation gives rise to additional structure that can be leveraged in both the representation of the energy functional and the design of efficient optimisation algorithms, as we now describe.

\subsubsection{Linear parameters and parametric basis functions}

To formalise this, we assume without loss of generality that the parameter space $\Theta$ decomposes as a product $\md{W} \times \md{X}$, where $\md{W} = \RR^{\dl}$ for some $\dl \geq 1$ represents a set of \emph{linear parameters} and $\md{X} \subset \RR^{\dnl}$, for some $\dnl \geq 0$, is a closed set of \emph{nonlinear parameters}. The realisation map then takes the form
$$\mc{R}(\mb{w}, \mb{\xi}) = \sum_{k = 1}^{\dl} \mb{w}_{k} \mb{\phi}_{k}(\mb{\xi}) = \mb{w}^{\ast} \mb{\phi}(\mb{\xi}),$$
where each $\mb{\phi}_{k}: \md{X} \to U$ is a \emph{parametric basis function}, gathered in the vector-valued map $\mb{\phi}: \md{X} \to U^{\dl}$. The vector $\mb{w} \in \md{W}$ thus plays the role of \emph{degrees of freedom}. Here $\mb{w}^{\ast} \mb{\phi}(\mb{\xi})$ denotes the inner product of two vectors.

This decomposition includes, as special cases, both purely linear and purely nonlinear parameterisations. Taking $\dnl = 0$ ($\md{X} = \emptyset$) recovers the classical setting of linear discretisation in a fixed basis. At the other extreme, a fully nonlinear model---where no linear parameter is separated---can be represented by setting $\dl = 1$ and constraining $\mb{w}_{1} = 1$, so that the realisation becomes $\mb{\phi}_{1}(\mb{\xi}) = \mc{R}(\mb{\xi})$.

When the approximation space separates linear parameters, it naturally acquires the structure of a vector bundle, in which each parameter corresponds to a vector space within the total space, known as its fiber. \fig{bundle} illustrates the key components of this structure: the base space, the total space, individual fibers, and sections that select one representative vector from each fiber.

\begin{figure}
    \centering
    % \includesvg[width=0.6\linewidth]{svg/bundle.svg}
    \def\svgwidth{0.6\linewidth}
    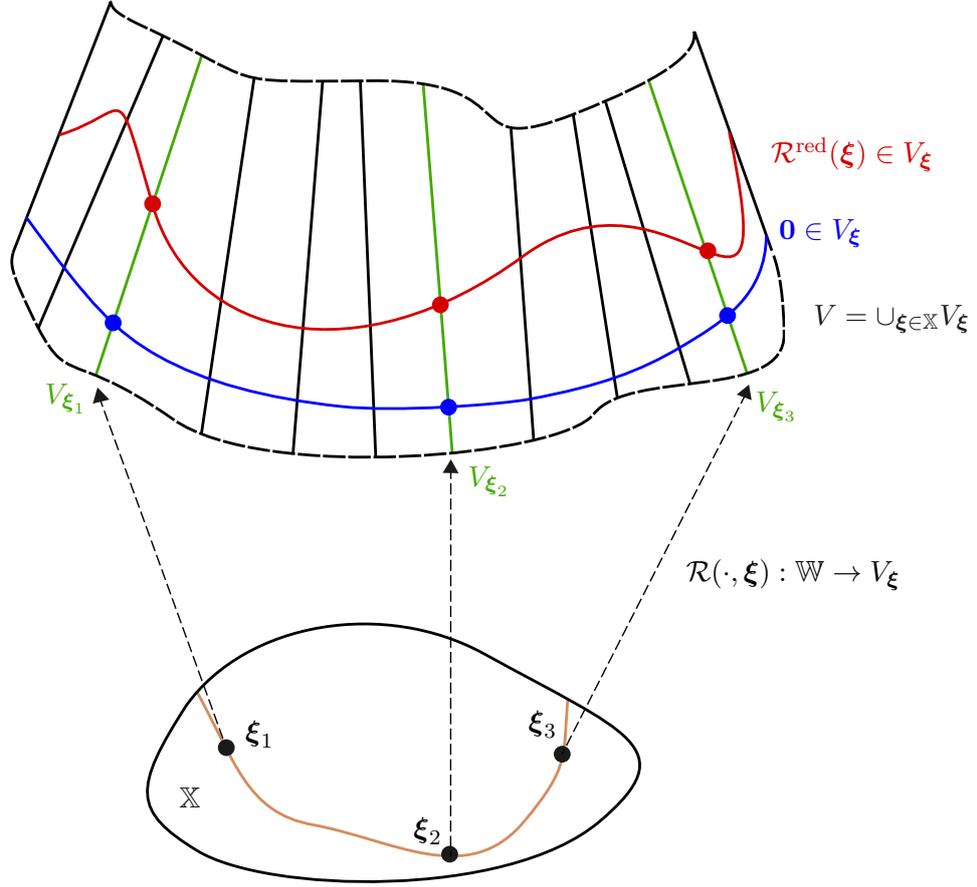
    \caption{Visualisation of a parameter-dependent approximation space with vector bundle structure. The base space $\md{X}$ (nonlinear parameter space) is shown at the bottom. To each parameter $\mb{\xi} \in \md{X}$ corresponds a fiber $V_{\mb{\xi}} = \mathrm{Span}(\mb{\phi}(\mb{\xi})) = \mc{R}(\md{W}, \mb{\xi})$. The figure depicts the fibers lying above parameters along the brown path in $\md{X}$, and highlights three fibers in green. Collectively, the fibers form the total space $V$, illustrating how the local linear spaces vary continuously with the parameters to assemble into a nonlinear approximation space. A section of this bundle is a map that assigns to every parameter $\mb{\xi}$ a vector in the corresponding fiber $V_{\mb{\xi}}$. Two sections are shown: $\red{\mc{R}}$, in red, which selects the energy minimiser in each fiber (see \eqref{eq:reduced-realisation}), and the zero section, in blue.}
    \label{fig:bundle}
\end{figure}

\subsubsection{Parametric quadratic energy}

This structure leads to a convenient form for the energy functional. Substituting the expression of $\mc{R}$ into the definition of $\mc{K}$, and using the bilinearity of $a$ and linearity of $\ell$, we obtain
$$\mc{K}(\mb{w}, \mb{\xi}) = \frac{1}{2} \mb{w}^{\ast} \mb{A}(\mb{\xi}) \mb{w} - \mb{w}^{\ast} \mb{\ell}(\mb{\xi}),$$
where the stiffness matrix $\mb{A}: \md{X} \to \RR^{\dl \times \dl}$ and the load vector $\mb{\ell}: \md{X} \to \RR^{\dl}$ are defined componentwise by
$$\mb{A}(\mb{\xi})_{ij} = a(\mb{\phi}_{j}(\mb{\xi}), \mb{\phi}_{i}(\mb{\xi})), \qquad \mb{\ell}(\mb{\xi})_{j} = \ell(\mb{\phi}_{j}(\mb{\xi})),$$
for all $i, j \in \rg{1}{\dl}$, respectively. Here the notation $\rg{a}{b}$, for $a \leq b \in \NN$, refers to the set of integers $\{a, \ldots, b\}$. By symmetry and coercivity of the bilinear form $a$, the matrix $\mb{A}(\mb{\xi})$ is symmetric and positive semi-definite for all $\mb{\xi} \in \md{X}$. It may fail to be positive definite if the basis functions $(\mb{\phi}_{k}(\mb{\xi}))_{k \in \rg{1}{\dl}}$ are linearly dependent.

For a fixed nonlinear parameter $\mb{\xi}$, the map $\mb{w} \mapsto \mc{K}(\mb{w}, \mb{\xi})$ is a convex quadratic function. More precisely, it is minimal when $\mb{w}$ solves the linear system
\begin{equation}
    \label{eq:linear-system}
    \mb{A}(\mb{\xi}) \mb{w} = \mb{\ell}(\mb{\xi}).
\end{equation}

This observation motivates treating the linear and nonlinear parameters differently in the optimisation process: the convexity in $\mb{w}$ can be exploited using solvers tailored to quadratic minimisation or even linear solvers to eliminate the linear parameters, while $\mb{\xi}$ may be updated using general-purpose nonlinear optimisation methods.

\subsubsection{Best linear parameters and reduced energy}
\label{subsect:reduced-energy}

The following lemma shows that the linear system defining the best linear parameters remains consistent even when the basis functions are linearly dependent. Here, consistent means that the linear system admits at least one solution; equivalently, $\mb{\ell}(\mb{\xi}) \in \operatorname{Im}(\mb{A}(\mb{\xi}))$.

\begin{lemma}
    \label{lem:consistency}
    The linear system \eqref{eq:linear-system} is consistent for all $\mb{\xi} \in \md{X}$. Moreover, any two solutions $\mb{w}_{1}, \mb{w}_{2} \in \md{W}$ define the same realisation in $V$; that is, $\mc{R}(\mb{w}_{1}, \mb{\xi}) = \mc{R}(\mb{w}_{2}, \mb{\xi})$.
\end{lemma}

\begin{proof}
    See \app{consistency}.
\end{proof}

In other words, \lem{consistency} ensures that given $\mb{\xi} \in \md{X}$, the energy functional $\mc{J}$ has a unique minimiser over the fibre
$V_{\mb{\xi}} \doteq \{\mc{R}(\mb{w}, \mb{\xi}), \mb{w} \in \md{W}\}$. The minimisation of $\mc{J}$ is even well-posed in $V_{\mb{\xi}}$, as
$$\alpha \norm{\mc{R}(\mb{w}, \mb{\xi})}_{U}^{2} \leq a(\mc{R}(\mb{w}, \mb{\xi}), \mc{R}(\mb{w}, \mb{\xi})) = \ell(\mc{R}(\mb{w}, \mb{\xi})) \leq \norm{\ell}_{U} \norm{\mc{R}(\mb{w}, \mb{\xi})}_{U},$$
so $\norm{\mc{R}(\mb{w}, \mb{\xi})}_{U} \leq \alpha^{-1} \norm{\ell}_{U}$ for all $\mb{w}$ solving the linear system \eqref{eq:linear-system}.

The ambiguity arising from a nontrivial kernel of $\mb{A}(\mb{\xi})$ can be resolved by selecting the unique solution of minimal Euclidean norm, which solves a convex quadratic problem: minimise the Euclidean norm $\norm{\mb{w}}_{2}$ over the affine solution space $\mb{w}_p + \operatorname{Ker}(\mb{A}(\mb{\xi}))$, where $\mb{w}_{p}$ is any particular solution to the linear system \eqref{eq:linear-system}. This problem has a unique solution, which corresponds to the Euclidean projection of $\mb{w}_{p}$ onto $\operatorname{Ker}(\mb{A}(\mb{\xi}))$. However, when $\mb{A}(\mb{\xi})$ is ill-conditioned, this projection may be numerically unstable, as its sensitivity increases with the inverse of the smallest nonzero eigenvalue of $\mb{A}(\mb{\xi})$; a quantity over which we have no direct control.

To circumvent this issue, we adopt a stronger assumption and require a uniform lower bound on the smallest eigenvalue of $\mb{A}(\mb{\xi})$. For clarity, we separate the conditioning of the bilinear form $a$ and that of the parametric basis $\mb{\phi}(\mb{\xi})$. To that aim, we introduce the Gram matrix $\mb{G}(\mb{\xi}) \in \RR^{\dl \times \dl}$ with entries
$$\mb{G}(\mb{\xi})_{ij} \doteq (\mb{\phi}_{j}(\mb{\xi}), \mb{\phi}_{i}(\mb{\xi}))_{U},$$
for all $i, j \in \rg{1}{\dl}$. This matrix satisfies $(\mc{R}(\mb{v}, \mb{\xi}), \mc{R}(\mb{w}, \mb{\xi}))_{U} = \mb{v}^{\ast} \mb{G}(\mb{\xi}) \mb{w}$ for all $\mb{v}, \mb{w} \in \md{W}$ and $\mb{\xi} \in \md{X}$. Let $\omega(\mb{\xi})$ denote the smallest eigenvalue of $\mb{G}(\mb{\xi})$. Then, the coercivity of $a$ implies
$$\mb{w}^{\ast} \mb{A}(\mb{\xi}) \mb{w} = a(\mc{R}(\mb{w}, \mb{\xi}), \mc{R}(\mb{w}, \mb{\xi})) \geq \alpha \norm{\mc{R}(\mb{w}, \mb{\xi})}_{U}^{2} = \alpha \mb{w}^{\ast} \mb{G}(\mb{\xi}) \mb{w} \geq \alpha \omega(\mb{\xi}) \norm{\mb{w}}_{2}^{2},$$
showing that the smallest eigenvalue of $\mb{A}(\mb{\xi})$ is greater than $\alpha \omega(\mb{\xi})$. Therefore, a uniform lower bound on $\omega(\mb{\xi})$ is sufficient to ensure the uniform positive definiteness of $\mb{A}(\mb{\xi})$.

\begin{assumption}
    \label{assumption:spd}
    There exists $\omega_{\min} > 0$ such that $\omega(\mb{\xi}) \geq \omega_{\min}$ for all $\mb{\xi} \in \md{X}$.
\end{assumption}

Under \assum{spd}, the linear system defining the best linear parameters can be solved uniquely at each nonlinear parameter, thereby defining the \emph{best linear parameters map} $\best{\mb{w}}: \md{X} \to \md{W}$ by
$$\best{\mb{w}}(\mb{\xi}) = \mb{A}(\mb{\xi})^{-1} \mb{\ell}(\mb{\xi}).$$
This induces the \emph{reduced realisation}
\begin{equation}
    \label{eq:reduced-realisation}
    \red{\mc{R}}: \md{X} \to V, \quad \mb{\xi} \mapsto \mc{R}(\best{\mb{w}}(\mb{\xi}), \mb{\xi}) = \mb{\ell}(\mb{\xi})^{\ast} \mb{A}(\mb{\xi})^{-1} \mb{\phi}(\mb{\xi}),
\end{equation}
and the associated \emph{reduced energy}
$$\red{\mc{K}}: \md{X} \to \RR, \quad \mb{\xi} \mapsto \mc{K}(\best{\mb{w}}(\mb{\xi}), \mb{\xi}) = -\frac{1}{2} \mb{\ell}(\mb{\xi})^{\ast} \mb{A}(\mb{\xi})^{-1} \mb{\ell}(\mb{\xi}).$$
Thus, minimising $\mc{K}$ over $\md{W} \times \md{X}$ is equivalent to minimising $\red{\mc{K}}$ over $\md{X}$.

The elimination the linear parameters from the optimisation problem reduces it to a purely nonlinear minimisation over $\md{X}$. Importantly, the gradient of the reduced energy can be computed without differentiating the best linear parameter map $\best{\mb{w}}$. Let $\DW$ and $\DX$ denote the gradients with respect to the linear and nonlinear parameters, respectively. Then, by construction,
$$\DW \mc{K}(\best{\mb{w}}(\mb{\xi}), \mb{\xi}) = \mb{0}.$$
Assuming that $\mc{K}$ is differentiable with respect to $\md{X}$, the chain rule implies that $\red{\mc{K}}$ is differentiable and
\begin{align*}
    \nabla \red{\mc{K}}(\mb{\xi}) & = \DW \mc{K}(\best{\mb{w}}(\mb{\xi}), \mb{\xi}) \DX \best{\mb{w}}(\mb{\xi}) + \DX \mc{K}(\mb{w}, \mb{\xi})|_{\mb{w} = \best{\mb{w}}(\mb{\xi})} \\
                                  & = \DX \mc{K}(\mb{w}, \mb{\xi})|_{\mb{w} = \best{\mb{w}}(\mb{\xi})}.
\end{align*}
The derivative of the map $\best{\mb{w}}$ does not appear in this expression. This simplification makes the reduced formulation especially attractive, as it enables a computationally efficient alternating optimisation strategy.

%% file: bundle_svg-tex.pdf_tex
%% Creator: Inkscape 1.4.2 (1:1.4.2+202505120737+ebf0e940d0), www.inkscape.org
%% Accompanies image file 'bundle_svg-tex.pdf' (pdf, eps, ps)
\begingroup%
  \makeatletter%
  \providecommand\color[2][]{%
    \errmessage{(Inkscape) Color is used for the text in Inkscape, but the package 'color.sty' is not loaded}%
    \renewcommand\color[2][]{}%
  }%
  \providecommand\transparent[1]{%
    \errmessage{(Inkscape) Transparency is used (non-zero) for the text in Inkscape, but the package 'transparent.sty' is not loaded}%
    \renewcommand\transparent[1]{}%
  }%
  \providecommand\rotatebox[2]{#2}%
  \newcommand*\fsize{\dimexpr\f@size pt\relax}%
  \newcommand*\lineheight[1]{\fontsize{\fsize}{#1\fsize}\selectfont}%
  \ifx\svgwidth\undefined%
    \setlength{\unitlength}{207bp}%
    \ifx\svgscale\undefined%
      \relax%
    \else%
      \setlength{\unitlength}{\unitlength * \real{\svgscale}}%
    \fi%
  \else%
    \setlength{\unitlength}{\svgwidth}%
  \fi%
  \global\let\svgwidth\undefined%
  \global\let\svgscale\undefined%
  \makeatother%
  \begin{picture}(1,1.14678225)%
    \lineheight{1}%
    \setlength\tabcolsep{0pt}%
    \put(0,0){\includegraphics[width=\unitlength,page=1]{bundle_svg-tex.pdf}}%
    \put(0.30327114,0.18536609){\makebox(0,0)[lt]{\lineheight{1.25}\smash{\begin{tabular}[t]{l}$\mb{\xi}_{1}$\end{tabular}}}}%
    \put(0.66831749,0.19814511){\makebox(0,0)[lt]{\lineheight{1.25}\smash{\begin{tabular}[t]{l}$\mb{\xi}_{3}$\end{tabular}}}}%
    \put(0.52006398,0.05898661){\makebox(0,0)[lt]{\lineheight{1.25}\smash{\begin{tabular}[t]{l}$\mb{\xi}_{2}$\end{tabular}}}}%
    \put(0.98106771,0.93664769){\color[rgb]{0.83137255,0,0}\makebox(0,0)[lt]{\lineheight{1.25}\smash{\begin{tabular}[t]{l}$\red{\mc{R}}(\mb{\xi}) \in V_{\mb{\xi}}$\end{tabular}}}}%
    \put(0.99154254,0.83784886){\color[rgb]{0,0,1}\makebox(0,0)[lt]{\lineheight{1.25}\smash{\begin{tabular}[t]{l}$\mb{0} \in V_{\mb{\xi}}$\end{tabular}}}}%
    \put(0,0){\includegraphics[width=\unitlength,page=2]{bundle_svg-tex.pdf}}%
    \put(0.0448701,0.62384885){\color[rgb]{0.26666667,0.66666667,0}\makebox(0,0)[lt]{\lineheight{1.25}\smash{\begin{tabular}[t]{l}$V_{\mb{\xi}_{1}}$\end{tabular}}}}%
    \put(0.59059314,0.51614982){\color[rgb]{0.26666667,0.66666667,0}\makebox(0,0)[lt]{\lineheight{1.25}\smash{\begin{tabular}[t]{l}$V_{\mb{\xi}_{2}}$\end{tabular}}}}%
    \put(0.96118772,0.61316858){\color[rgb]{0.26666667,0.66666667,0}\makebox(0,0)[lt]{\lineheight{1.25}\smash{\begin{tabular}[t]{l}$V_{\mb{\xi}_{3}}$\end{tabular}}}}%
    \put(0.21788846,0.09838248){\color[rgb]{0.10196078,0.10196078,0.10196078}\makebox(0,0)[lt]{\lineheight{1.25}\smash{\begin{tabular}[t]{l}$\md{X}$\end{tabular}}}}%
    \put(1.03595203,0.72614211){\color[rgb]{0.10196078,0.10196078,0.10196078}\makebox(0,0)[lt]{\lineheight{1.25}\smash{\begin{tabular}[t]{l}$V = \cup_{\mb{\xi} \in \md{X}} V_{\mb{\xi}}$\end{tabular}}}}%
    \put(0.87198849,0.39280507){\makebox(0,0)[lt]{\lineheight{1.25}\smash{\begin{tabular}[t]{l}$\mc{R}(\cdot, \mb{\xi}): \md{W} \to V_{\mb{\xi}}$\end{tabular}}}}%
    \put(0,0){\includegraphics[width=\unitlength,page=3]{bundle_svg-tex.pdf}}%
  \end{picture}%
\endgroup%

%% file: 3.algorithm.tex
Building on the structure of the abstract framework, we introduce a hybrid optimisation algorithm tailored to approximation spaces with both linear and nonlinear parameters.

\subsection{Two-step algorithm}

We propose a minimisation algorithm that generates a sequence of parameters $(\mb{w}_{k}, \mb{\xi}_{k})_{k \geq 0}$ by alternating updates of the linear and nonlinear parameters, leveraging convexity in $\mb{w}$ and accommodating the generally nonconvex nature in $\mb{\xi}$. A generic version of the alternating scheme is summarised in \alg{optimiser}.

The algorithm begins by updating the linear parameters based on the initial nonlinear ones (line 1), ensuring that the linear variables are optimally (or approximately) set before any nonlinear updates. At each iteration, the nonlinear parameters are first updated using a general optimisation step (line 4). We present and analyse the mirror descent update \eqref{eq:nonlinear-step}, but other optimisers could be employed. After this, the linear parameters are updated based on the new nonlinear parameters (line 5). We consider two strategies for this step: either an exact solve of the linear system \eqref{eq:full-linear-solve}, or an approximate update that guarantees a decrease in energy \eqref{eq:partial-linear-solve}. The algorithm concludes with a final linear solve (line 13), typically carried out with higher accuracy than the updates performed during the optimisation loop, to ensure that the linear parameters are close to optimal.

Owing to the conformity of the approximation, making use of \eqref{eq:energy-gap}, the discrete energy may be expressed as
$$\mc{K}(\mb{w}, \mb{\xi}) = \mc{J}(u^{\star}) + \frac{1}{2} \norm{\mc{R}(\mb{w}, \mb{\xi}) - u^{\star}}_{a}^{2}.$$
Thus, minimising $\mc{K}(\mb{w}, \mb{\xi})$ is equivalent to minimising distance between the approximation $\mc{R}(\mb{w}, \mb{\xi})$ and the exact solution $u^{\star}$ in the energy norm. This interpretation motivates tracking the parameters $(\mb{w}_{\min}, \mb{\xi}_{\min})$ corresponding to the lowest energy observed during the optimisation process (lines 6--8), and returning the associated approximation $\mc{R}(\mb{w}_{\min}, \mb{\xi}_{\min})$ at the end of the training (line 14). This mechanism also offers worst-case bounds by guaranteeing that the algorithm produces a solution at least as good as the initial one.

The algorithm terminates after a fixed number of iteration, or when an early-stopping criterion is met (lines 9--11). Common criteria include the stabilisation of nonlinear parameters, indicated by
$$\norm{\mb{\xi}_{k+1} - \mb{\xi}_{k}}_{\md{X}} \leq \epsilon_{\md{X}}$$
for some threshold $\epsilon_{\md{X}} > 0$, or the plateauing of the energy values, detected when
$$|\mc{K}(\mb{w}_{k+1}, \mb{\xi}_{k+1}) - \mc{K}(\mb{w}_{k}, \mb{\xi}_{k})| \leq \epsilon_{\mc{K}}$$
for some threshold $\epsilon_{\mc{K}} > 0$. Here $\norm{\cdot}_{\md{X}}$ denotes a norm on $\RR^{\dnl}$.

As demonstrated in \subsect{local-convergence}, the distance between successive parameters provides an upper bound on the norm of the gradient of the energy functional. This makes it a suitable criterion for identifying a quasi-minimiser. Other termination criteria could be used, such as requiring the slope of the energy values over a fixed window to fall below a prescribed threshold, indicating that further progress is negligible.

\begin{algorithm}
    \caption{Two-step alternating minimisation of the energy functional. Given initial parameters $(\mb{w}_{0}, \mb{\xi}_{0}) \in \md{W} \times \md{X}$, number of epochs $E \geq 0$, linear and nonlinear update routines, step sizes $(\gamma_{k})_{k \geq 0}$ and tolerances $\epsilon_{\md{X}}, \epsilon_{\mc{K}} > 0$.}
    \label{alg:optimiser}
    \begin{algorithmic}[1]
        \State $\mb{w}_{0} \gets \texttt{UpdateLinear}(\mb{w}_{0}, \mb{\xi}_{0})$ \Comment{Initial update of linear parameters}
        \State $(\mb{w}_{\min}, \mb{\xi}_{\min}, \mc{K}_{\min}) \gets (\mb{w}_{0}, \mb{\xi}_{0}, \mc{K}(\mb{w}_{0}, \mb{\xi}_{0}))$
        \For{$k = 0, \ldots, E - 1$}
        \State $\mb{\xi}_{k+1} \gets \texttt{UpdateNonlinear}(\mb{w}_{k}, \mb{\xi}_{k}, \gamma_{k})$ \Comment{Update nonlinear parameters}
        \State $\mb{w}_{k+1} \gets \texttt{UpdateLinear}(\mb{w}_{k}, \mb{\xi}_{k+1})$ \Comment{Update linear parameters}
        \If{$\mc{K}(\mb{w}_{k+1}, \mb{\xi}_{k+1}) < \mc{K}_{\min}$} \Comment{Track best parameters}
        \State $(\mb{w}_{\min}, \mb{\xi}_{\min}, \mc{K}_{\min}) \gets (\mb{w}_{k+1}, \mb{\xi}_{k+1}, \mc{K}(\mb{w}_{k+1}, \mb{\xi}_{k+1}))$
        \EndIf
        \If{$\norm{\mb{\xi}_{k+1} - \mb{\xi}_k}_{\md{X}} \leq \epsilon_{\md{X}}$ or $\abs{\mc{K}(\mb{w}_{k+1}, \mb{\xi}_{k+1}) - \mc{K}(\mb{w}_{k}, \mb{\xi}_{k})} \leq \epsilon_{\mc{K}}$} \Comment{Check convergence}
        \State \textbf{break}
        \EndIf
        \EndFor
        \State $\mb{w}_{\min} \gets \texttt{UpdateLinear}(\mb{w}_{\min}, \mb{\xi}_{\min})$ \Comment{Final update of linear parameters}
        \State \Return $\mc{R}(\mb{w}_{\min}, \mb{\xi}_{\min})$
    \end{algorithmic}
\end{algorithm}

\subsection{Update of the linear parameters}

Since the matrix $\mb{A}(\mb{\xi})$ is symmetric and positive definite under \assum{spd}, the best linear parameters can be efficiently computed using the \ac{cg} method \cite{saad2003iterative}. A natural update for the linear component in \alg{optimiser} is therefore
\begin{equation}
    \label{eq:full-linear-solve}
    \texttt{UpdateLinear}(\mb{w}, \mb{\xi}) = \mb{A}(\mb{\xi})^{-1} \mb{\ell}(\mb{\xi}),
\end{equation}
with the inverse implicitly computed via \ac{cg}. This is particularly advantageous when $\mb{A}(\mb{\xi})$ is large and sparse.

When a full linear solve is too expensive, it can be replaced with an inexact update. For instance, one may perform a fixed number of \ac{cg} steps starting from the previous iterate $\mb{w}_{k}$, or take a few gradient descent steps to reduce the energy. A simple instance of the latter is the steepest descent update:
\begin{equation}
    \label{eq:partial-linear-solve}
    \texttt{UpdateLinear}(\mb{w}, \mb{\xi}) = \mb{w} + \beta(\mb{\xi}) \mb{r}(\mb{\xi}),
\end{equation}
where $\mb{r}(\mb{\xi}) \doteq \mb{\ell}(\mb{\xi}) - \mb{A}(\mb{\xi}) \mb{w} \in \md{W}^{\ast}$ (the dual of $\md{W}$) and the step size is given by
$$\beta(\mb{\xi}) = \frac{(\mb{r}(\mb{\xi}), \mb{r}(\mb{\xi}))_{2}}{(\mb{r}(\mb{\xi}), \mb{A}(\mb{\xi}) \mb{r}(\mb{\xi}))_{2}} \geq 0.$$
Here $(\cdot, \cdot)_{2}$ denotes the Euclidean inner product in $\md{W}$. Strictly speaking, the update $\mb{w} + \beta(\mb{\xi}) \mb{r}(\mb{\xi})$ is ill-posed at the continuous level, since $\mb{r}(\mb{\xi}) \in \md{W}^{\ast}$ and $\mb{w} \in \md{W}$. A proper formulation would require applying a Riesz isomorphism to map $\mb{r}(\mb{\xi})$ back into the primal space. However, in this finite-dimensional setting, we identify $\md{W} \cong \md{W}^{\ast}$ via the Euclidean structure and take the Riesz map to be the identity.

\subsection{Update of the nonlinear parameters}

A variety of methods are available for updating the nonlinear parameters, including momentum-based schemes such as ADAM \cite{kingma2014adam} or RMSProp, quasi-Newton methods like BFGS and SR1 \cite[Chapter 6]{nocedal1999numerical}, and metric-aware approaches such as natural gradients \cite{amari1998natural, martens2020new}. Each offers different trade-offs in terms of convergence speed, robustness, and computational cost. For clarity of presentation and ease of analysis, we focus here on mirror gradient descent, a flexible framework that naturally accommodates constraints and includes projected gradient descent as a special case.

In our setting, the parameter set $\md{X}$ is constrained, and the optimisation step must incorporate a projection back onto $\md{X}$. Depending on the structure of the feasible set $\md{X}$, computing the Euclidean projection onto $\md{X}$ required by projected gradient descent can be computationally demanding or even intractable. Mirror descent offers a principled generalisation by replacing the Euclidean distance in the proximal step with a Bregman divergence generated by a strictly convex distance-generating function \cite{bubeck2015convex}.

Let $\psi: \md{X} \to \RR$ be a twice differentiable, $\mu$-strongly convex function with respect to $\norm{\cdot}_{\md{X}}$, meaning that
$$\psi(\mb{\eta}) \geq \psi(\mb{\xi}) + \inner{\nabla \psi(\mb{\xi})}{\mb{\eta} - \mb{\xi}}_{\md{X}} + \frac{1}{2} \mu \norm{\mb{\eta} - \mb{\xi}}_{\md{X}}^{2},$$
for all $\mb{\xi}, \mb{\eta} \in \md{X}$. Here $\inner{\cdot}{\cdot}_{\md{X}}$ denotes the duality pairing between $\md{X}$ (seen as a subset of $\RR^{\dnl}$) and the dual of $\RR^{\dnl}$. The associated \emph{Bregman divergence} centred at $\mb{\xi} \in \md{X}$ is given by
$$D_{\psi}(\mb{\eta}; \mb{\xi}) \doteq \psi(\mb{\eta}) - \psi(\mb{\xi}) - \inner{\nabla \psi(\mb{\xi})}{\mb{\eta} - \mb{\xi}}_{\md{X}}.$$
It can be shown that the function $D_{\psi}(\ \cdot \ ; \mb{\xi})$ is nonnegative and $\mu$-strongly convex for every fixed $\mb{\xi} \in \md{X}$.

The \emph{proximal map} $\operatorname{Prox}: \md{W} \times \md{X} \times \RR_{+} \rightrightarrows \md{X}$ is defined by
$$\proxmap{\mb{w}}{\mb{\xi}}{\gamma} = \arg \min_{\mb{\eta} \in \md{X}} \gamma \inner{\DX \mc{K}(\mb{w}, \mb{\xi})}{\mb{\eta}}_{\md{X}} + D_{\psi}(\mb{\eta}; \mb{\xi}).$$
The proximal map may be set-valued if $\md{X}$ is not convex. The \emph{mirror descent update} is then obtained by choosing a suitable step size $\gamma > 0$ and setting
\begin{equation}
    \label{eq:nonlinear-step}
    \texttt{UpdateNonlinear}(\mb{w}, \mb{\xi}, \gamma) \in \proxmap{\mb{w}}{\mb{\xi}}{\gamma}.
\end{equation}
This step balances descent in the direction $-\DX \mc{K}(\mb{w}, \mb{\xi})$ with proximity to the current iterate $\mb{\xi}$, measured in the geometry induced by $\psi$.

The mirror descent update requires minimising a strongly convex function over the nonlinear parameter space. The existence readily follows from the closedness of $\md{X}$, but the uniqueness requires that $\md{X}$ be closed. Altogether, we make the following assumption to ensure the existence of the mirror descent step. We also include the compactness of the parameter space to ensure the existence and numerical reachability of global minimisers, by continuity of the discrete energy and the extreme value theorem.

\begin{assumption}
    \label{assumption:compact-convex}
    The nonlinear parameter space $\md{X}$ is convex and compact.
\end{assumption}

A notable special case of mirror descent is when $\psi(\mb{\xi}) = \norm{\mb{\xi}}_{\md{X}}^{2} / 2$, which yields $D_{\psi}(\mb{\eta}; \mb{\xi}) = \norm{\mb{\eta} - \mb{\xi}}_{\md{X}}^{2} / 2$ and the proximal map $\proxmap{\mb{w}}{\mb{\xi}}{\gamma}$ reduces to the standard Euclidean projection onto $\md{X}$, recovering projected gradient descent.

%% file: 4.convergence.tex
We now analyse the convergence properties of the hybrid optimisation algorithm introduced above. Our results are divided into local and global components, depending on the strength of the structural assumptions placed on the reduced energy functional.

Local convergence is established under mild regularity conditions, such as Hölder continuity of the gradient, which are satisfied in many practical settings. Global convergence requires stronger assumptions, notably directional convexity of the reduced energy in a neighbourhood of the global minimisers. Under these conditions, we prove that the algorithm converges to a quasi-stationary point within the basin of attraction of a global minimum.

The resulting bounds yield a nonlinear analogue of Céa's lemma, showing that the computed solution approximates the best possible within the chosen approximation space, up to a quantifiable optimisation error.

To clarify the logical structure of our analysis and the dependencies between assumptions and results, we conclude this section with a diagram outlining the main implications, see \fig{roadmap}. Each node represents a key assumption or result, and directed edges indicate logical dependencies. This visual summary provides a high-level overview of the proof strategy and serves as a reference for verifying which hypotheses are required for each convergence result.

The diagram also highlights two layers of assumptions in our analysis. In practice, \assum{spd} and \assum{linear-boundedness-regularity} are typically straightforward to verify, as they reflect standard properties of well-posed discretisations. Likewise, \assum{compact-convex} holds for many parameter spaces. In contrast, \assum{nonlinear-boundedness-regularity} may fail in certain approximation spaces, especially when nonlinear parametrisations yield irregular basis functions. Even then, it is often possible to verify \assum{regularity} directly. This regularity property is the key requirement for establishing the local convergence of our algorithm and can often be checked independently of the more structural assumptions. The example discussed in \subsect{differentiability} illustrates such a situation.

\begin{figure}
    \centering
    \begin{tikzpicture}[
            scale=0.85, transform shape,
            node distance=2cm, thick,
            assumption/.style={
                    draw=blue!70!black,
                    fill=blue!5,
                    rounded corners=5pt,
                    thick,
                    text=blue!70!black,
                    font=\sffamily\small,
                    minimum width=3.5cm,
                    minimum height=1.5cm,
                    align=center,
                    rectangle,
                },
            result/.style={
                    draw=gray!50!black,
                    fill=gray!10,
                    thick,
                    font=\itshape\small\sffamily,
                    minimum width=3.5cm,
                    minimum height=1.5cm,
                    align=center,
                    rectangle,
                },
        ]
        \node[assumption] (A1) {\assum{spd} \\ $\mb{A}(\mb{\xi}) \succcurlyeq \alpha \omega_{\min}$};
        \node[assumption] (A3) [right=1cm of A1] {\assum{linear-boundedness-regularity} \\ $\mb{\phi}$ bounded \\ and Hölder};
        \node[assumption] (A2) [below=1cm of A1] {\assum{compact-convex} \\ $\md{X}$ convex \\ and compact};
        \node[assumption] (A5) [below=1cm of A3] {\assum{directional-convexity} \\ Directional \\ convexity};
        \node[assumption] (A4) [right=1cm of A5] {\assum{regularity} \\ $\DX \mc{K}$ Hölder};
        \node[assumption] (A6) [right=1cm of A3] {\assum{nonlinear-boundedness-regularity} \\ $\mb{\phi}$ diff. in $U$, $\DX \mb{\phi}$ \\ bounded and Hölder};

        \draw[ultra thick] (A1.north) ++(-2, 0.2) -- ++(13.1, 0) -- ++(0, -4.5) -- ++(-13.1, 0) -- ++ (0, 4.53);

        \draw[->, >=latex, align=center] (A6.east) ++(0.3, 0) to[out=-20, in=20, looseness=3] node[above=5mm, font=\itshape\small\sffamily] {\lem{regularity-energy} \\ (\ref*{assumption:linear-boundedness-regularity}, \ref*{assumption:nonlinear-boundedness-regularity})} (A4.east);
        \draw[dashed, thick] (A1.north) ++(11.1, -2) -- ++(-4.1, 0) -- ++(0, 2.2);

        \node[result] (L2) [below=1cm of A2] {\lem{energy-decrease} \\ Sufficient decrease};
        \node[result] (L3) [below=1cm of L2] {\lem{local-convergence} \\ Local convergence};
        \node[result] (L4) [right=1cm of L3] {\lem{surrogate} \\ Stopping criterion};
        \draw[->, >=latex] (L2.north) ++(0, 0.73) -- node[right=1mm, font=\itshape\small\sffamily] {(\ref*{assumption:linear-boundedness-regularity})} (L2.north);
        \draw[->, >=latex] (L2) -- (L3);
        \draw[->, >=latex] (L4.north) ++(0, 3.3) -- node[right=1mm, font=\itshape\small\sffamily] {(\ref*{assumption:compact-convex}, \ref*{assumption:regularity})} (L4.north);
        \draw[->, >=latex] (L3.west) ++(-0.25, 4.5) -- +(-1, 0) |- node[left=1mm, font=\itshape\small\sffamily] {(\ref*{assumption:compact-convex}, \ref*{assumption:regularity})} (L3);

        \node[result] (L7) [below=1cm of A4] {\lem{bounds-best-linear} \\ $\nabla \red{\mc{K}}$ Hölder};
        \node[result] (L5) [below=1cm of L7] {\lem{global-convergence}, \corol{nonlinear-cea} \\ Global convergence, \\ Nonlinear Céa's lemma};
        \draw[->, >=latex] (L7) -- (L5);
        \draw[->, >=latex] (L7.north) ++(0, 0.73) -- node[right=1mm, font=\itshape\small\sffamily] {(\ref*{assumption:spd}, \ref*{assumption:linear-boundedness-regularity}, \ref*{assumption:regularity})} (L7.north);
        \draw[->, >=latex] (L5.east) ++(0.08, 4.5) -- +(1, 0) |- node[right=1mm, font=\itshape\small\sffamily] {(\ref*{assumption:compact-convex}, \ref*{assumption:directional-convexity})} (L5);
    \end{tikzpicture}
    \caption{Logical dependency graph between assumptions and results in our abstract framework. Nodes represent assumptions (blue) and results (gray). In situations where \assum{nonlinear-boundedness-regularity} is not satisfied, \assum{regularity} may be checked independently.}
    \label{fig:roadmap}
\end{figure}
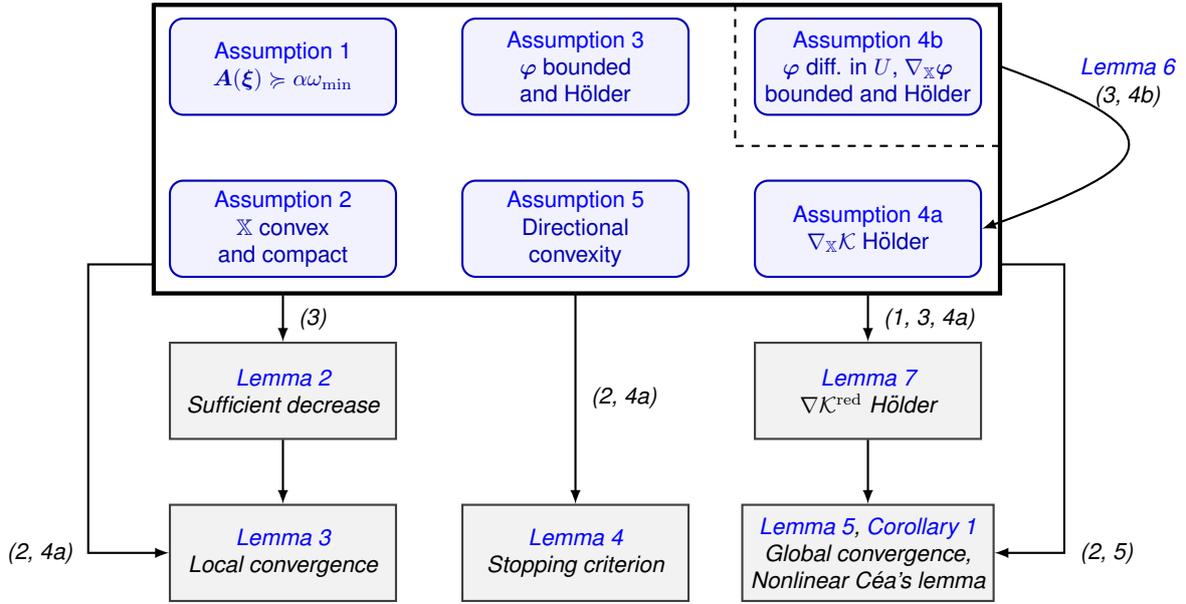

\subsection{Local convergence of mirror descent}
\label{subsect:local-convergence}

The local convergence of \alg{optimiser} relies on a guaranteed energy decrease at each step: for the linear parameters, this follows from classical properties of quadratic minimisation; for the nonlinear parameters, it is ensured by regularity of the energy gradient with respect to the nonlinear variables.

\subsubsection{Decrease condition for the linear parameters}

Both the full linear solve \eqref{eq:full-linear-solve} and the steepest descent update \eqref{eq:partial-linear-solve} satisfy a decrease condition, as formulated in the following lemma.

\begin{lemma}
    \label{lem:energy-decrease}
    Let $\mb{w} \in \md{W}$ and $\mb{\xi} \in \md{X}$, and define $\mb{w}_{+} = \texttt{UpdateLinear}(\mb{w}, \mb{\xi})$ via the update rule \eqref{eq:full-linear-solve} or \eqref{eq:partial-linear-solve}. It holds
    \begin{equation}
        \label{eq:decrease-cgsd}
        \mc{K}(\mb{w}_{+}, \mb{\xi}) \leq \mc{K}(\mb{w}, \mb{\xi}) - \frac{1}{2} \lambda_{\max}(\mb{\xi})^{-1} \norm{\DW \mc{K}(\mb{w}, \mb{\xi})}_{2, \ast}^{2},
    \end{equation}
    where $\lambda_{\max}(\mb{A}(\mb{\xi})) \leq \norm{a}_{U \times U} \norm{\mb{\phi}(\mb{\xi})}_{U, 2}^{2}$ is the largest eigenvalue of $\mb{A}(\mb{\xi})$.
\end{lemma}

\begin{proof}
    See \app{energy-decrease}
\end{proof}

In the upper bound of the largest eigenvalue, the norm $\norm{\cdot}_{2, \ast}$ denotes dual norm associated to $\norm{\cdot}_{2}$, and $\norm{\cdot}_{U, 2}$ denotes the norm on $U^{\dl}$ defined by
$$\norm{\mb{\phi}(\mb{\xi})}_{U, 2}^{2} \doteq \sum_{i = 1}^{\dl} \norm{\mb{\phi}_{i}(\mb{\xi})}_{U}^{2}.$$

The local convergence of our algorithm will require the uniform boundedness of the constant in front of $\norm{\DW \mc{K}(\mb{w}, \mb{\xi})}_{2, \ast}^{2}$ in the energy decrease after the update of the linear parameters. In other words, it requires the uniform boundedness of $\norm{\mb{\phi}(\mb{\xi})}_{U, 2}$, as formalised in the following assumption. For later use, we also include the regularity of the basis functions in the same assumption.

\begin{assumption}
    \label{assumption:linear-boundedness-regularity}
    The basis functions are uniformly bounded and Hölder continuous in $\md{X}$; that is,
    \begin{enumerate}
        \item $\norm{\mb{\phi}}_{U, 2, \infty} \doteq \sup_{\mb{\xi} \in \md{X}} \norm{\mb{\phi}(\mb{\xi})}_{U, 2} < \infty$,
        \item there exist $L_{\mb{\phi}} > 0$ and $\nu \in \oc{0}{1}$ such that $\norm{\mb{\phi}(\mb{\xi}) - \mb{\phi}(\mb{\eta})}_{U, 2} \leq L_{\mb{\phi}} \norm{\mb{\xi} - \mb{\eta}}_{\md{X}}^{\nu}$ for all $\mb{\xi}, \mb{\eta} \in \md{X}$.
    \end{enumerate}
\end{assumption}

Under \assum{linear-boundedness-regularity}, the decrease condition \eqref{eq:decrease-cgsd} holds uniformly over $\md{X}$ with constant $\lambda_{\max}(\mb{A}(\mb{\xi}))^{-1} \geq \norm{a}_{U \times U}^{-1} \norm{\mb{\phi}}_{U, 2, \infty}^{-2}$. This ensures that each update of the linear parameters yields an energy decrease uniformly proportional to the squared norm of the energy gradient.

\begin{remark}
    \assum{linear-boundedness-regularity} does not generally hold for neural networks with standard activation functions such as $\tanh$ or ReLU. For example, the $H^{1}$ semi-norm of the function $x \mapsto \tanh(a x + b)$ on a bounded interval grows proportionally with $\abs{a}^{1/2}$, so the assumption can only hold if the parameter $a$ is uniformly bounded. This observation would motivate bounding the parameter space of a neural network.
\end{remark}

\begin{remark}
    More generally, the local convergence of \alg{optimiser} can be established for any update rule of the linear parameters that satisfies the following energy decrease condition: for all $\mb{\xi} \in \md{X}$, there exists $0 < \beta(\mb{\xi}) < \infty$ such that
    \begin{equation}
        \label{eq:decrease}
        \mc{K}(\mb{w}_{+}, \mb{\xi}) \leq \mc{K}(\mb{w}, \mb{\xi}) - \frac{1}{2} \beta(\mb{\xi})^{-1} \norm{\DW \mc{K}(\mb{w}, \mb{\xi})}_{2, \ast}^{2}
    \end{equation}
    for all $\mb{w} \in \md{W}$ and $\mb{w}_{+} = \texttt{UpdateLinear}(\mb{w}, \mb{\xi})$. This condition is stated in such a way that $\beta(\mb{\xi})$ has the same physical dimension as an eigenvalue of $\mb{A}(\mb{\xi})$. A uniform lower bound for $\beta(\mb{\xi})$ is easily obtained from \assum{spd}: the optimal update being $\best{\mb{w}}(\mb{\xi})$, we obtain
    \begin{align*}
        \mc{K}(\mb{w}_{+}, \mb{\xi}) - \mc{K}(\mb{w}, \mb{\xi}) & \geq \mc{K}(\best{\mb{w}}(\mb{\xi}), \mb{\xi}) - \mc{K}(\mb{w}, \mb{\xi})                                     \\
                                                                & = - \frac{1}{2} (\best{\mb{w}}(\mb{\xi}) - \mb{w})^{\ast} \mb{A}(\mb{\xi}) (\best{\mb{w}}(\mb{\xi}) - \mb{w}) \\
                                                                & = - \frac{1}{2} \DW \mc{K}(\mb{w}, \mb{\xi})^{\ast} \mb{A}(\mb{\xi})^{-1} \DW \mc{K}(\mb{w}, \mb{\xi})        \\
                                                                & \geq - \frac{1}{2} \lambda_{\min}(\mb{A}(\mb{\xi}))^{-1} \norm{\DW \mc{K}(\mb{w}, \mb{\xi})}_{2, \ast}^{2}.
    \end{align*}
    Combining this fact with \eqref{eq:decrease}, we find $\beta(\mb{\xi}) \geq \lambda_{\min}(\mb{A}(\mb{\xi}))$. Using \assum{spd}, we reach $\beta(\mb{\xi}) \geq \alpha \omega_{\min}$. As \lem{energy-decrease} showed, obtaining a uniform upper bound for $\beta(\mb{\xi})$ typically requires the uniform boundedness of $\mb{\phi}(\mb{\xi})$ in $U$.
\end{remark}

\subsubsection{Local convergence}

We recall the convergence properties of mirror descent. Although the results are classical, we state them under slightly relaxed assumptions to account for the convexity of the energy functional in the linear parameters and for the possibility that the set of global minimisers has positive measure (it may not consist of isolated points). Our analysis builds on and extends the basin-based convergence framework developed in~\cite{traonmilin2023basins}, which provides guarantees for gradient descent in non-convex optimisation. We adapt these ideas to the constrained setting, where the energy is minimised via mirror descent rather than standard gradient descent, by incorporating tools from constrained optimisation.

Standard analyses of gradient descent schemes typically require some form of gradient regularity, such as Lipschitz continuity~\cite[Chapter 1]{polyak1987introduction} or Hölder continuity~\cite{yashtini2016global}. Since the discrete energy is quadratic in the linear parameters, we only need to assume locally Hölder gradients with respect to the nonlinear parameters.

\renewcommand{\theassumption}{4a}

\begin{assumption}
    \label{assumption:regularity}
    For all $\mb{w} \in \md{W}$, the map $\md{X} \ni \mb{\xi} \mapsto \mc{K}(\mb{w}, \mb{\xi})$ has Hölder continuous gradients; that is, there exists $\nu \in \oc{0}{1}$ such that for each $\mb{w} \in \md{W}$, there exists $L(\mb{w}) > 0$ satisfying
    $$\norm{\DX \mc{K}(\mb{w}, \mb{\xi}) - \DX \mc{K}(\mb{w}, \mb{\eta})}_{\md{X}, \ast} \leq L(\mb{w}) \norm{\mb{\xi} - \mb{\eta}}_{\md{X}}^{\nu}$$
    for all $\mb{\xi}, \mb{\eta} \in \md{X}$.
\end{assumption}

Here $\norm{\cdot}_{\md{X}, \ast}$ denotes the dual norm induced by $\norm{\cdot}_{\md{X}}$.

\begin{remark}
    More general moduli of gradient continuity could be considered, such as log-Lipschitz regularity, Dini-continuity or other concave modulus functions. We restrict our analysis to Hölder conditions for simplicity, which cover most cases of practical interest.
\end{remark}

\renewcommand{\theassumption}{\arabic{assumption}}

To preserve the clarity of the exposition, we defer the discussion of sufficient conditions for \assum{regularity} to the end of the section. As shown in \lem{regularity-energy}, the local Hölder constant $L(\mb{w})$ needs to depend on $\mb{w}$, unless $\mb{w}$ is uniformly bounded.

In particular, \assum{regularity} implies that $\mc{K}$ is continuous. Since $\md{X}$ is closed under \assum{compact-convex} and $\md{W}$ is closed, the product space $\md{W} \times \md{X}$ is also closed. Given that $\mc{J}$ is also coercive, the Weierstrass theorem guarantees the existence of a global minimum of $\mc{K}$ in $\md{W} \times \md{X}$~\cite{ekeland1999convex}. Let $\mc{K}^{\star} \doteq \min_{\mb{w} \in \md{W}, \mb{\xi} \in \md{X}} \mc{K}(\mb{w}, \mb{\xi})$ denote the value of that minimum.

We now introduce common tools for constrained optimisation on convex sets. The \emph{normal cone} to $\md{X}$ at $\mb{\xi} \in \md{X}$ is defined as
$$N_{\md{X}}(\mb{\xi}) \doteq \{\mb{v} \in (\RR^{\dnl})^{\ast} \ : \ \forall \mb{\eta} \in \md{X}, \inner{\mb{v}}{\mb{\eta} - \mb{\xi}}_{\md{X}} \leq 0\}.$$
Normal cones are instrumental in expressing first-order necessary optimality conditions over convex sets: if $(\mb{w}, \mb{\xi}) \in \md{W} \times \md{X}$ is a local minimum of $\mc{K}$, then $\DW \mc{K}(\mb{w}, \mb{\xi}) = \mb{0}$ and $-\DX \mc{K}(\mb{w}, \mb{\xi}) \in N_{\md{X}}(\mb{\xi})$. This condition, however, is not sufficient for local minimality, since it also characterises saddle points and local maxima. Nevertheless, since $D_{\psi}$ is strongly convex, it does not have saddle point or local maxima so any $\mb{\xi}_{+} \in \proxmap{\mb{w}}{\mb{\xi}}{\gamma}$ is characterised by the first-order optimality condition
\begin{equation}
    \label{eq:def-prox}
    -\gamma \DX \mc{K}(\mb{w}, \mb{\xi}) - \nabla \psi(\mb{\xi}_{+}) + \nabla \psi(\mb{\xi}) \in N_{\md{X}}(\mb{\xi}_{+}).
\end{equation}

For convenience, we also define the \emph{gradient mapping} $G_{\psi}: \md{W} \times \md{X} \times \RR_{+} \to (\RR^{\dnl})^{\ast}$ via
$$\gradmap{\mb{w}}{\mb{\xi}}{\gamma} \doteq \gamma^{-1} R_{\md{X}}(\mb{\xi} - \texttt{UpdateNonlinear}(\mb{w}, \mb{\xi}, \gamma)),$$
where $R_{\md{X}}: \RR^{\dnl} \to (\RR^{\dnl})^{\ast}$ is the Riesz map defined by $\inner{R_{\md{X}} \mb{\xi}}{\mb{\eta}}_{\md{X}} = (\mb{\xi}, \mb{\eta})_{\md{X}}$ for all $\mb{\xi}, \mb{\eta} \in \RR^{\dnl}$. In this way, one can write
$$\texttt{UpdateNonlinear}(\mb{w}, \mb{\xi}, \gamma) = \mb{\xi} - \gamma R_{\md{X}}^{-1} \gradmap{\mb{w}}{\mb{\xi}}{\gamma},$$
where $R_{\md{X}}^{-1}: (\RR^{\dnl})^{\ast} \to \RR^{\dnl}$ is the inverse Riesz map. This formulation provides a consistent way to interpret the gradient mapping as a dual vector.

In particular, in the case of projected gradient descent, when the unconstrained step $\mb{\xi} - \gamma R_{\md{X}}^{-1} \DX \mc{K}(\mb{w}, \mb{\xi})$ remains inside $\md{X}$, the proximal map $\proxmap{\mb{w}}{\mb{\xi}}{\gamma}$ reduces to $\mb{\xi} - \gamma R_{\md{X}}^{-1} \DX \mc{K}(\mb{w}, \mb{\xi})$, and $\gradmap{\mb{w}}{\mb{\xi}}{\gamma} = \DX \mc{K}(\mb{w}, \mb{\xi})$. Thus, $G_{\psi}$ plays the role of an effective gradient, incorporating both the projection and step size effects.

We now show the local convergence of the mirror descent iterates to a quasi-minimiser.

\begin{lemma}[Local convergence]
    \label{lem:local-convergence}
    Let $\mb{w}, \mb{\xi} \in \md{W} \times \md{X}$ and $(\mb{w}_{+}, \mb{\xi}_{+}) \in \md{W} \times \md{X}$ denote the next iterate produced by \alg{optimiser} with step size $\gamma > 0$. Under \assum{compact-convex} and \assum{regularity}, for all $\epsilon > 0$, or $\epsilon = 0$ if $\nu = 1$, it holds
    $$\mc{K}(\mb{w}_{+}, \mb{\xi}_{+}) \leq \mc{K}(\mb{w}, \mb{\xi}) - \frac{1}{2} \gamma (2 \mu - \gamma L_{\nu, \epsilon}(\mb{w})) \norm{\gradmap{\mb{w}}{\mb{\xi}}{\gamma}}_{\md{X}, \ast}^{2} - \frac{1}{2} \beta(\mb{\xi})^{-1} \norm{\DW \mc{K}(\mb{w}, \mb{\xi}_{+})}_{\md{W}, \ast}^{2} + \epsilon,$$
    where $L_{\nu, \epsilon}(\mb{w}) \doteq \left(\frac{1}{2 \epsilon} \frac{1 - \nu}{1 + \nu}\right)^{\frac{1 - \nu}{1 + \nu}} L(\mb{w})^{\frac{2}{1 + \nu}}$ and $\beta(\mb{\xi}) \doteq \lambda_{\max}(\mb{A}(\mb{\xi}))$.

    Let $(\mb{w}_{0}, \mb{\xi}_{0}) \in \md{W} \times \md{X}$ denote some initial parameters and $(\mb{w}_{k}, \mb{\xi}_{k})_{k \geq 0}$ denote the sequence of parameters produced by \alg{optimiser} with step sizes $\gamma_{k} > 0$. Under the same assumptions, for all $n \geq 1$, it holds
    $$\min_{0 \leq k \leq n-1} \left\{\norm{\gradmap{\mb{w}_{k}}{\mb{\xi}_{k}}{\gamma_{k}}}_{\md{X}, \ast}^{2} + \norm{\DW \mc{K}(\mb{w}_{k}, \mb{\xi}_{k+1})}_{\md{W}, \ast}^{2}\right\} \leq \frac{2 (\mc{K}(\mb{w}_{0}, \mb{\xi}_{0}) - \mc{K}^{\star} + n \epsilon)}{S_{n}},$$
    where $S_{n} \doteq \sum_{k = 0}^{n-1} \min(\gamma_{k} (2 \mu - \gamma_{k} L_{\nu, \epsilon}(\mb{w}_{k})), \beta(\mb{\xi}_{k})^{-1})$ provided $S_{n} > 0$.
\end{lemma}

\begin{proof}
    See \app{local-convergence}.
\end{proof}

In the case of Lipschitz gradients ($\nu = 1$), taking $\epsilon = 0$ ($L_{\nu, \epsilon} = L$) in \lem{local-convergence} shows that under a suitable step size, the value of the energy functional is non-increasing along the trajectory of \alg{optimiser}. In the more general setting where the gradient is only Hölder continuous, the energy functional need not decrease monotonically during \alg{optimiser}. Under \subassum{linear-boundedness-regularity}{1}, we have $\beta(\mb{\xi}) \leq \beta_{\max} \doteq \norm{a}_{U \times U} \norm{\mb{\phi}}_{U, 2, \infty}$. Using a step size $\gamma_{k} = \zeta \mu / L_{\nu, \epsilon}(\mb{w}_{k})$ for some $\zeta \in \oo{0}{2}$ and assuming $L(\mb{w}) \leq L_{\max}$, \lem{local-convergence} still provides
$$\min_{0 \leq k \leq n-1} \left\{\norm{\gradmap{\mb{w}_{k}}{\mb{\xi}_{k}}{\gamma_{k}}}_{\md{X}, \ast}^{2} + \norm{\DW \mc{K}(\mb{w}_{k}, \mb{\xi}_{k+1})}_{\md{W}, \ast}^{2}\right\} \leq \frac{2 (\mc{K}(\mb{w}_{0}, \mb{\xi}_{0}) - \mc{K}^{\star} + n \epsilon)}{n \min(\mu^{2} \zeta (2 - \zeta) L_{\nu, \epsilon, \max}^{-1}, \beta_{\max}^{-1})},$$
where $L_{\nu, \epsilon, \max} = \left(\frac{1}{2 \epsilon} \frac{1 - \nu}{1 + \nu}\right)^{\frac{1 - \nu}{1 + \nu}} L_{\max}^{\frac{2}{1 + \nu}}$. As $n \to \infty$, this upper bound converges to
$$\frac{2 \left(\frac{1}{2} \frac{1 - \nu}{1 + \nu}\right)^{\frac{1 - \nu}{1 + \nu}} L_{\max}^{\frac{2}{1 + \nu}} \epsilon^{\frac{2 \nu}{1 + \nu}}}{\min\left(\mu^{2} \zeta (2 - \zeta), \left(\frac{1}{2 \epsilon} \frac{1 - \nu}{1 + \nu}\right)^{\frac{1 - \nu}{1 + \nu}} L_{\max}^{\frac{2}{1 + \nu}} \beta_{\max}^{-1}\right)},$$
where we expanded the expression of $L_{\nu, \epsilon, \max}$ to show the dependence of this bound on $\epsilon$. As $\epsilon \to 0$, the denominator becomes $\mu^{2} \zeta (2 - \zeta)$, while the numerator goes to $0$, so this upper bound can be made arbitrarily small by decreasing $\epsilon$.

Importantly, the estimate of \lem{local-convergence} shows that it is not necessary to solve the linear system exactly at each iteration for \alg{optimiser} to converge. Any linear update that ensures the energy decrease condition in~\eqref{eq:decrease} suffices to guarantee convergence.

\begin{remark}
    The proof of \lem{local-convergence} fundamentally relies on the fact that a gradient descent step decreases the energy whenever the gradient is nonzero. More precisely, the mirror descent update \eqref{eq:nonlinear-step} satisfies
    $$\mc{K}(\mb{w}, \mb{\xi}_{+}) \leq \mc{K}(\mb{w}, \mb{\xi}) - \frac{1}{2} \gamma (2 \mu - \gamma L_{\nu, \epsilon}(\mb{w})) \norm{\gradmap{\mb{w}}{\mb{\xi}}{\gamma}}_{\md{X}, \ast}^{2} + \epsilon.$$
    The convergence proof extends directly to any update rule for the nonlinear parameters that guarantees a similar energy decrease.
\end{remark}

\subsection{Identification of quasi-stationary points}

In practical settings, reaching an exact stationary point is often too costly or unnecessary in light of numerical error and diminishing returns, and it is often more reasonable to look for a quasi-stationary point.

Given $\epsilon > 0$, we say that $(\mb{w}, \mb{\xi}) \in \md{W} \times \md{X}$ is an \emph{$\epsilon$-quasi-stationary point} if $\norm{\DW \mc{K}(\mb{w}, \mb{\xi})}_{2, \ast} \leq \epsilon$ and $-\DX \mc{K}(\mb{w}, \mb{\xi}) \in N_{\md{X}}(\mb{\xi}) + \overline{B}_{\md{X}, \ast}(0, \epsilon)$, where
$$N_{\md{X}}(\mb{\xi}) + \overline{B}_{\md{X}, \ast}(0, \epsilon) \doteq \{\mb{v} \in (\RR^{p})^{\ast} \ : \ \exists \mb{w} \in N_{\md{X}}(\mb{\xi}), \norm{\mb{v} - \mb{w}}_{\md{X}, \ast} \leq \epsilon\}$$
is an $\epsilon$-neighbourhood of $N_{\md{X}}(\mb{\xi})$. The following result shows that the gradient map acts as a surrogate for quasi-stationarity.

\begin{lemma}[Surrogate for quasi-stationarity]
    \label{lem:surrogate}
    Under \assum{compact-convex}, let $(\mb{w}, \mb{\xi}) \in \md{W} \times \md{X}$, and $\mb{\xi}_{+}$ denote the next iterate produced by \alg{optimiser} with step $\gamma > 0$. Under \assum{regularity}, $(\mb{w}, \mb{\xi}_{+})$ is a $(L (\gamma c)^{\nu} + \mu c)$-quasi stationary point, where $c^{2} = \norm{\gradmap{\mb{w}}{\mb{\xi}}{\gamma}}_{\md{X}, \ast}^{2} + \norm{\DW \mc{K}(\mb{w}, \mb{\xi}_{+})}_{\md{W}, \ast}^{2}$.
\end{lemma}

\begin{proof}
    See \app{surrogate}.
\end{proof}

In other words, \lem{surrogate} states that the condition
$$\norm{\gradmap{\mb{w}}{\mb{\xi}}{\gamma}}_{\md{X}, \ast}^{2} + \norm{\DW \mc{K}(\mb{w}, \mb{\xi}_{+})}_{\md{W}, \ast}^{2} \leq \epsilon^{2}$$
can serve as a termination criterion to find a quasi-stationary point during the minimisation process. This termination criterion is equivalent to
$$\gamma^{-2} \norm{\mb{\xi}_{+} - \mb{\xi}}_{\md{X}}^{2} + \norm{\DW \mc{K}(\mb{w}, \mb{\xi}_{+})}_{\md{W}, \ast}^{2} \leq \epsilon^{2}.$$
Since a final linear solve is performed at the end of the optimisation process (\alg{optimiser}, line 13), we do not need to check the stabilisation of the linear parameters. In that case, the stopping criterion becomes $\norm{\mb{\xi}_{+} - \mb{\xi}}_{\md{X}} \leq \epsilon \gamma$, as implemented in \alg{optimiser} with $\epsilon_{\md{X}} = \epsilon \gamma$.

We now combine \lem{local-convergence} and \lem{surrogate} to estimate the number of steps to find a quasi-stationary point in the Lipschitz case ($\nu = 1$, $\epsilon = 0$). Running \alg{optimiser} with step size $\gamma_{k} = \zeta \mu / L_{\nu, \epsilon}(\mb{w}_{k})$ for some $\zeta \in \oo{0}{2}$ and assuming $L(\mb{w}) \leq L_{\max}$ and $\beta(\mb{\xi}) \leq \beta_{\max}$ (\subassum{linear-boundedness-regularity}{1}) will produce an iterate $(\mb{w}, \mb{\xi})$ such that $\norm{\gradmap{\mb{w}}{\mb{\xi}}{\gamma}}_{\md{X}, \ast}^{2} + \norm{\DW \mc{K}(\mb{w}, \mb{\xi}_{+})}_{\md{W}, \ast}^{2} \leq c^{2}$ after at most $2 (\mc{K}(\mb{w}_{0}, \mb{\xi}_{0}) - \mc{K}^{\star}) / [c^{2} \min(\mu^{2} \zeta (2 - \zeta) L_{\max}^{-1}, \beta_{\max}^{-1})]$ iterations. Given a tolerance $\tau > 0$, choose $c > 0$ such that $\gamma L_{\max} c + \mu c = \tau$. Then a $\tau$-quasi-stationary point is guaranteed after at most
$$E(\tau) = \frac{2 \mu^{2} (1 + \zeta)^{2} (\mc{K}(\mb{w}_{0}, \mb{\xi}_{0}) - \mc{K}^{\star})}{\tau^{2} \min(\mu^{2} \zeta (2 - \zeta) L_{\max}^{-1}, \beta_{\max}^{-1})}$$
iterations. This number of iterations is proportional to $2 (\mc{K}(\mb{w}_{0}, \mb{\xi}_{0}) - \mc{K}^{\star})$, which is equal to $\norm{\mc{R}(\mb{w}_{0}, \mb{\xi}_{0}) - \mc{R}^{\star}}_{a}^{2}$, the squared energy distance between the initial realisation and a global minimiser $\mc{R}^{\star}$ of $\mc{J}$ in $V$. The prefactor depending on $\zeta$ is minimal for the choice
$$\zeta = \begin{cases}
        1 - \sqrt{1 - \frac{L_{\max}}{\mu^{2} \beta_{\max}}} & \text{if } \frac{L_{\max}}{\mu^{2} \beta_{\max}} \leq \frac{3}{4}, \\
        \frac{1}{2}                                          & \text{otherwise}.
    \end{cases}$$

\begin{remark}
    Our analysis highlights the critical role of the Lipschitz constant of the gradient of the discrete energy functional, as it directly constrains the maximum step size in \alg{optimiser}. A large Lipschitz constant necessitates smaller steps, potentially slowing down progress towards the minimiser.

    This limitation is structural to first-order methods based on fixed metrics. However, it may be alleviated through more advanced optimisation techniques that incorporate curvature information. For instance, natural gradient methods~\cite{amari1998natural, nurbekyan2023efficient} or its low-rank approximations~\cite{bioli2025accelerating} adapt the step direction using the inverse of the metric tensor induced by an inner product of the embedding space $U$ (for example, the one induced by the bilinear form of the variational problem), while quasi-Newton approaches approximate second-order information to adjust step sizes dynamically. Such strategies could lead to stronger convergence guarantees under less restrictive conditions on the step size, especially in settings where the gradient of the discrete energy functional has a large Lipschitz constant.
\end{remark}

\subsection{Global convergence of projected gradient descent}

We now turn to global convergence guarantees for \alg{optimiser}, focusing on the case where the reduced energy functional admits Lipschitz continuous gradients. This is the case whenever both \assum{spd}, \assum{linear-boundedness-regularity} and \assum{regularity} hold. For clarity of exposition, we postpone the verification of this sufficient condition to the next subsection.

A notable difficulty in the analysis of the basins of attraction of discrete energies defined on nonlinear approximation spaces is related to the lack of injectivity of the realisation map. Indeed, there may exist $v \in V$ such that the inverse image $\mc{R}^{-1}(\{v\})$ is not a set of isolated points. Consequently, the set of global minimisers of $\mc{K}$ may have a positive measure in $\md{W} \times \md{X}$, and $\mc{K}$ may fail to be convex in a neighbourhood of one of its global minimisers.

We will prove the existence of basins of attraction for the reduced energy functional under the assumption of directional convexity in a neighbourhood of its global minimisers.

To define directional convexity, we introduce a few notations and definitions. Using \assum{spd}, any pair $(\mb{w}, \mb{\xi}) \in \md{W} \times \md{X}$ satisfying $\mc{K}(\mb{w}, \mb{\xi}) = \mc{K}^{\star}$ must be such that $\mb{w} = \best{\mb{w}}(\mb{\xi})$. It follows that the reduced energy $\red{\mc{K}}$ also attains its minimum on $\md{X}$, and that this minimum agrees with $\mc{K}^{\star}$. We define the set of \emph{best nonlinear parameters} as
$$\md{X}^{\star} \doteq \{\mb{\xi} \in \md{X} \ : \ \red{\mc{K}}(\mb{\xi}) = \mc{K}^{\star}\},$$
the \emph{distance to the best nonlinear parameters},
$$\delta_{\psi}^{\star}: \md{X} \to \RR_{+}, \qquad \delta_{\psi}^{\star}(\mb{\xi}) \doteq \inf_{\mb{\xi}^{\star} \in \md{X}^{\star}} D_{\psi}(\mb{\xi}^{\star}; \mb{\xi}),$$
and the set-valued \emph{projection onto $\md{X}^{\star}$},
$$\Pi_{\psi}^{\star}: \md{X} \rightrightarrows \md{X}^{\star}, \qquad \Pi_{\psi}^{\star}(\mb{\xi}) \doteq \arg \inf_{\mb{\xi}^{\star} \in \md{X}^{\star}} D_{\psi}(\mb{\xi}^{\star}; \mb{\xi}).$$
Directional convexity is defined as follows.

\begin{definition}[Directional convexity]
    We say that $\red{\mc{K}}$ is \emph{directionally convex} in a subset $S \subset \md{X}$ if for all $\mb{\xi} \in S$, there exists $\mb{\xi}^{\star} \in \Pi_{\psi}^{\star}(\mb{\xi})$ such that $\red{\mc{K}}$ is convex on the segment $[\mb{\xi}, \mb{\xi}^{\star}]$.
\end{definition}

Directional convexity is weaker than standard convexity. For example, the function $f: \RR^{2} \ni (x, y) \mapsto (x^{2} + y^{2} - 1)^{2}$ reaches its minimum on the unit circle. For the generating function $\psi(x, y) = \norm{(x, y)}_{2}^{2} / 2$, one can check that $\Pi_{\psi}^{\star}(x, y) = (x, y) / \norm{(x, y)}_{2}$, and $f$ is directionally convex in the region $\{(x, y) \in \RR^{2}, x^{2} + y^{2} \geq 1/3\}$. However, $f$ is convex only in the region $\{(x, y) \in \RR^{2}, x^{2} + y^{2} \geq 1\}$.

We state a global convergence result for the iterates of \alg{optimiser} under the following directional convexity assumption.
\begin{assumption}
    \label{assumption:directional-convexity}
    There exists $\rho > 0$ such that $\red{\mc{K}}$ is directionally convex in the neighbourhood
    $$\Lambda_{\psi}^{\star}(\rho) \doteq \{\mb{\xi} \in \md{X} \ : \ \delta_{\psi}^{\star}(\mb{\xi}) \leq \rho\}.$$
\end{assumption}
Noticing that $D_{\psi}^{\star}(\mb{\xi}) = 0$ if and only if $\mb{\xi} \in \md{X}^{\star}$, the neighbourhood $\Lambda_{\psi}^{\star}(\rho)$ corresponds to the subset of nonlinear parameters that are at distance at most $\rho$ from $\md{X}^{\star}$.

\begin{lemma}[Global convergence]
    \label{lem:global-convergence}
    Under \assum{compact-convex}, \assum{directional-convexity} and assuming that $\red{\mc{K}}$ is Lipschitz continuous with constant $\red{L} > 0$, let $\mb{\xi}_{0} \in \Lambda_{\psi}^{\star}(\rho)$ and define the sequence $(\mb{\xi}_{k})_{k \geq 0}$ by running \alg{optimiser} with exact updates of the linear parameters, and step sizes $(\gamma_{k})_{k \geq 0}$ satisfying $\gamma_{k} \leq \mu / \red{L}$ for all $k \geq 0$. Then
    \begin{itemize}
        \item The distance to the best nonlinear parameters is non-increasing: $\delta_{\psi}^{\star}(\mb{\xi}_{k+1}) \leq \delta_{\psi}^{\star}(\mb{\xi}_{k})$ for all $k \geq 0$ so that $\mb{\xi}_{k} \in \Lambda_{\psi}^{\star}(\delta_{\psi}^{\star}(\mb{\xi}_{0}))$ for all $k \geq 0$, and
        \item The reduced energy satisfies the bound
              $$\red{\mc{K}}(\mb{\xi}_{n}) \leq \mc{K}^{\star} + \frac{\delta_{\psi}^{\star}(\mb{\xi}_{0})}{\sum_{k = 0}^{n-1} \gamma_{k}}$$
              for all $n \geq 1$.
    \end{itemize}
\end{lemma}
\begin{proof}
    See \app{global-convergence}.
\end{proof}

Under suitable assumptions on the energy functional, \lem{global-convergence} guarantees the stability of mirror descent iterates within a neighbourhood of the global minimisers. In particular, the distance to the best nonlinear parameters decreases monotonically along the trajectory, and the reduced energy converges to the global minimum. These properties define a \emph{basin of attraction} for the global minimisers under the mirror descent dynamics.

\begin{remark}
    Our proof of \lem{global-convergence} remains valid for more general star-shaped neighbourhoods $S$ of $\md{X}^{\star}$ satisfying the following property: if $\mb{\xi} \in S$ and $\mb{\eta} \in \md{X}$ is such that $\delta_{\psi}^{\star}(\mb{\eta}) \leq \delta_{\psi}^{\star}(\mb{\xi})$, then $\mb{\eta} \in S$. In particular, $S$ can be of the type $S = \{\mb{\xi} \in \md{X} : \exists \mb{\xi}^{\star} \in \md{X}^{\star}, D_{\psi}(\mb{\xi}^{\star}; \mb{\xi}) \leq \rho(\mb{\xi}^{\star})\}$, where $\mb{\xi}^{\star} \mapsto \rho(\mb{\xi}^{\star}) > 0$ is some parameter-dependent radius.
\end{remark}

\begin{remark}
    Various generalisations of convexity can be used to derive global convergence guarantees for mirror descent. We refer to \cite{karimi2016linear} for a comprehensive discussion and precise definitions of conditions such as (weak) strong convexity, the restricted secant inequality, the \acf{pl} inequality, and quadratic growth. Among these, the \ac{pl} condition is particularly appealing in practice, as it guarantees global convergence without requiring convexity, by ensuring that the gradient norm controls the suboptimality; that is, the difference between the current energy value and the global minimum. In this work, we focus on directional convexity as a structural assumption tailored to our problem setting, but similar arguments could be developed under other gradient growth conditions like \ac{pl}.
\end{remark}

\begin{remark}
    Even when the initialisation lies outside the basin of attraction, global convergence may still be achieved if the iterates eventually enter the basin. Once this occurs, the stability properties of the basin guarantee that all subsequent iterates remain within it, leading to global convergence. For instance, it is shown in~\cite{jin2017escape} that under the assumption of a Lipschitz-continuous Hessian, a perturbed gradient descent algorithm can escape strict saddle points and enter a basin of attraction within $O(\epsilon^{-2})$ iterations, matching the order of complexity we derived above.
\end{remark}

In our case, the energy functional is directly related to the distance to the global minimiser $u^{\star}$ of $\mc{J}$ (in the energy norm) via the relation
$$\red{\mc{K}}(\mb{\xi}) = \mc{J}(\red{\mc{R}}(\mb{\xi})) = \mc{J}(u^{\star}) + \frac{1}{2} \norm{\red{\mc{R}}(\mb{\xi}) - u^{\star}}_{a}^{2}.$$
The following corollary is a direct consequence of \lem{global-convergence}.
\begin{corollary}
    \label{corol:nonlinear-cea}
    With the same assumptions and notations as in \lem{global-convergence}, choosing step sizes $\gamma_{k} = \zeta \mu / \red{L}$ for some $\zeta \in \oc{0}{1}$, it holds
    $$\norm{\red{\mc{R}}(\mb{\xi}_{n}) - u^{\star}}_{a}^{2} \leq \inf_{v \in V} \norm{v - u^{\star}}_{a}^{2} + \frac{2 \red{L} \delta_{\psi}^{\star}(\mb{\xi}_{0})}{\zeta \mu n},$$
    for all $n \geq 1$.
\end{corollary}

\begin{proof}
    See \app{nonlinear-cea}.
\end{proof}

\corol{nonlinear-cea} can be understood as a nonlinear version of Céa's lemma: $\red{\mc{R}}(\mb{\xi}_{n})$ approaches an optimal solution in the approximation space $V$ at a rate of $O(n^{-1/2})$, where $n$ denotes the iteration count. Crucially, this provides a quantitative and rigorous guarantee that the discrete solution generated by \alg{optimiser} approaches the optimal energy-minimising solution within the nonlinear approximation space.

\subsection{Sufficient conditions for convergence}

To simplify the application of our framework, we provide sufficient conditions for the key assumptions, formulated in terms of boundedness and regularity properties of the parametric basis functions.

\subsubsection{Regularity of the energy}

We show the regularity of the gradient of the discrete energy with respect to $\md{W}$ and $\md{X}$ when the realisation is differentiable in $U$, as formalised in the following assumption. For convenience, we include the boundedness and regularity of the gradient of the realisation in this assumption. We equip $U^{\dnl \times \dl}$ with the norm $\norm{\mb{U}}_{U, 2, 2}^{2} \doteq \sum_{i = 1}^{\dnl} \sum_{j = 1}^{\dl} \norm{\mb{U}_{i, j}}_{U}^{2}$ for all $\mb{U} \in U^{\dnl \times \dl}$.

\renewcommand{\theassumption}{4b}

\begin{assumption}
    \label{assumption:nonlinear-boundedness-regularity}
    The basis functions are differentiable in $U$ and their gradient is uniformly bounded and Hölder continuous in $\md{X}$; that is,
    \begin{enumerate}
        \item $\partial_{i} \mb{\phi}_{k}(\mb{\xi}) \in U$, for all $\mb{\xi} \in \md{X}$, $k \in \rg{1}{\dl}$ and $i \in \rg{1}{\dnl}$,
        \item $\norm{\DX \mb{\phi}}_{U, 2, 2, \infty} \doteq \sup_{\mb{\xi} \in \md{X}} \norm{\DX \mb{\phi}(\mb{\xi})}_{U, 2, 2} < \infty$,
        \item there exist $L_{\DX \mb{\phi}} > 0$ and $\nu \in \oc{0}{1}$ such that $\norm{\DX \mb{\phi}(\mb{\xi}) - \DX \mb{\phi}(\mb{\eta})}_{U, 2, 2} \leq L_{\DX \mb{\phi}} \norm{\mb{\xi} - \mb{\eta}}_{\md{X}}^{\nu}$ for all $\mb{\xi}, \mb{\eta} \in \md{X}$.
    \end{enumerate}
\end{assumption}

\renewcommand{\theassumption}{\arabic{assumption}}

\begin{lemma}
    \label{lem:regularity-energy}
    For all $\mb{v}, \mb{w} \in \md{W}$, $\mb{\xi}, \mb{\eta} \in \md{X}$,
    \begin{align*}
        \norm{\DW \mc{K}(\mb{v}, \mb{\xi}) - \DW \mc{K}(\mb{w}, \mb{\eta})}_{2} & \leq \norm{a}_{U \times U} M_{\mb{\phi}}(\mb{\xi}, \mb{\eta})^{2} \norm{\mb{v} - \mb{w}}_{2}                                                                          \\
                                                                                & + (2 \norm{a}_{U \times U} M_{\md{W}}(\mb{v}, \mb{w}) M_{\mb{\phi}}(\mb{\xi}, \mb{\eta}) + \norm{\ell}_{U}) \norm{\mb{\phi}(\mb{\xi}) - \mb{\phi}(\mb{\eta})}_{U, 2},
    \end{align*}
    where $M_{\mb{\phi}}(\mb{\xi}, \mb{\eta}) = \max(\norm{\mb{\phi}(\mb{\xi})}_{U, 2}, \norm{\mb{\phi}(\mb{\eta})}_{U, 2})$ and $M_{\md{W}}(\mb{v}, \mb{w}) = \max(\norm{\mb{v}}_{2}, \norm{\mb{w}}_{2})$.

    Under \subassum{nonlinear-boundedness-regularity}{1}, for all $\mb{v}, \mb{w} \in \md{W}$, $\mb{\xi}, \mb{\eta} \in \md{X}$,
    \begin{align*}
         & \norm{\DX \mc{K}(\mb{v}, \mb{\xi}) - \DX \mc{K}(\mb{w}, \mb{\eta})}_{2}                                                                                                                                   \\
         & \leq (2 \norm{a}_{U \times U} M_{\md{W}}(\mb{v}, \mb{w}) M_{\mb{\phi}}(\mb{\xi}, \mb{\eta}) + \norm{\ell}_{U}) M_{\DX \mb{\phi}}(\mb{\xi}, \mb{\eta}) \norm{\mb{v} - \mb{w}}_{2}                          \\
         & + \norm{a}_{U \times U} M_{\md{W}}(\mb{v}, \mb{w})^{2} M_{\DX \mb{\phi}}(\mb{\xi}, \mb{\eta}) \norm{\mb{\phi}(\mb{\xi}) - \mb{\phi}(\mb{\eta})}_{U, 2}                                                    \\
         & + M_{\md{W}}(\mb{v}, \mb{w}) (\norm{a}_{U \times U} M_{\md{W}}(\mb{v}, \mb{w}) M_{\mb{\phi}}(\mb{\xi}, \mb{\eta}) + \norm{\ell}_{U}) \norm{\DX \mb{\phi}(\mb{\xi}) - \DX \mb{\phi}(\mb{\eta})}_{U, 2, 2},
    \end{align*}
    where $M_{\DX \mb{\phi}}(\mb{\xi}, \mb{\eta}) = \max(\norm{\DX \mb{\phi}(\mb{\xi})}_{U, 2, 2}, \norm{\DX \mb{\phi}(\mb{\eta})}_{U, 2, 2})$.
\end{lemma}

\begin{proof}
    See \app{regularity-energy}.
\end{proof}

In particular, the map $\DW \mc{K}$ is Lipschitz continuous with respect to the linear parameters provided that the basis function map $\mb{\phi}$ is uniformly bounded in $U$ (\subassum{linear-boundedness-regularity}{1}). With respect to the nonlinear parameters, $\DW \mc{K}$ inherits the modulus of continuity of $\mb{\phi}$, but this continuity holds only locally in the linear parameters (on bounded subsets). The Lipschitz regularity of $\DW \mc{K}$ in the linear parameters is only needed when the linear parameters are not separated from the nonlinear parameters. If they are separated, only the regularity in the nonlinear parameters is needed, at every fixed linear parameter (see \assum{regularity}).

For the derivative $\DX \mc{K}$, the continuity properties follow from those of $\mb{\phi}$ and its parametric derivative $\DX \mb{\phi}$, again only locally in the linear parameters. This observation motivated the formulation of \assum{regularity} allowing the continuity constant to depend on the norm of the linear parameters.

It is clear from \lem{regularity-energy} that \assum{linear-boundedness-regularity} and \assum{nonlinear-boundedness-regularity} together imply \assum{regularity}. However, when the basis functions fail to be differentiable in $U$, this implication breaks down; yet \assum{regularity} may still hold and can often be verified directly.

\subsubsection{Regularity of the reduced energy}
\label{subsect:regularity-reduced-energy}

It was shown in \subsect{reduced-energy} that the lowest eigenvalue of $\mb{A}(\mb{\xi})$ is larger than $\alpha \omega(\mb{\xi})$. Besides, \lem{energy-decrease} shows that the largest eigenvalue of $\mb{A}(\mb{\xi})$ is smaller than $\norm{a}_{U \times U} \norm{\mb{\phi}(\mb{\xi})}_{U, 2}^{2}$. Under \assum{spd} and \subassum{linear-boundedness-regularity}{1}, the condition number of $\mb{A}(\mb{\xi})$ is thus uniformly bounded by
$$\kappa_{\max} \doteq \frac{\norm{a}_{U \times U}}{\alpha} \frac{\norm{\mb{\phi}}_{U, 2, \infty}}{\omega_{\min}}.$$
With this notation, we can establish the boundedness of the best linear parameters and the regularity of the reduced energy.

\begin{lemma}
    \label{lem:bounds-best-linear}
    Under \assum{spd} and \subassum{linear-boundedness-regularity}{1}, it holds
    $$\sup_{\mb{\xi} \in \md{X}} \norm{\best{\mb{w}}(\mb{\xi})}_{2} \leq \frac{\norm{\ell}_{U} \norm{\mb{\phi}}_{U, 2, \infty}}{\alpha \omega_{\min}},$$
    and for all $\mb{\xi}, \mb{\eta} \in \md{X}$, we have the bound
    $$\norm{\best{\mb{w}}(\mb{\xi}) - \best{\mb{w}}(\mb{\eta})}_{2} \leq (1 + 2 \kappa_{\max}) \frac{\norm{\ell}_{U}}{\alpha \omega_{\min}} \norm{\mb{\phi}(\mb{\xi}) - \mb{\phi}(\mb{\eta})}_{U, 2}.$$
    Under the additional \subassum{nonlinear-boundedness-regularity}{1} and \subassum{nonlinear-boundedness-regularity}{2}, it also holds
    \begin{align*}
        \norm{\nabla \red{\mc{K}}(\mb{\xi}) - \nabla \red{\mc{K}(\mb{\eta})}}_{2} & \leq [\kappa_{\max} + (1 + 2 \kappa_{\max})^{2}] \frac{\norm{\ell}_{U}^{2}}{\alpha \omega_{\min}} \norm{\DX \mb{\phi}}_{U, 2, 2, \infty} \norm{\mb{\phi}(\mb{\xi}) - \mb{\phi}(\mb{\eta})}_{U, 2} \\
                                                                                  & + (1 + \kappa_{\max}) \frac{\norm{\ell}_{U}^{2}}{\alpha \omega_{\min}} \norm{\mb{\phi}}_{U, 2, \infty} \norm{\DX \mb{\phi}(\mb{\xi}) - \DX \mb{\phi}(\mb{\eta})}_{U, 2, 2}.
    \end{align*}
\end{lemma}

\begin{proof}
    See \app{bounds-best-linear}.
\end{proof}

\lem{bounds-best-linear} establishes that the best linear parameters are uniformly bounded and inherit the same regularity as the basis functions. Consequently, the reduced energy functional inherits the continuity properties of the basis functions and their gradients. Whenever $\DX \mc{K}$ is Lipschitz (via \assum{regularity} or \lem{regularity-energy}), the fact that $\best{\mb{w}}$ is Lipschitz and the multilinearity of $\DX \mc{K}$ in $\mb{w}$ ensure that $\nabla \red{\mc{K}}$ is Lipschitz.
When \assum{nonlinear-boundedness-regularity} holds, \lem{bounds-best-linear} provides tighter constants.

\subsubsection{Directional convexity}

We conclude our analysis with a lower bound on the Hessian of $\red{\mc{K}}$, which can be leveraged to show \assum{directional-convexity}.

\begin{lemma}
    \label{lem:convexity}
    Suppose $\mb{\phi}(\mb{\xi})$ is twice differentiable in $U$ for all $\mb{\xi} \in \md{X}$. For all $\mb{\xi} \in \md{X}$, $\mb{v} \in \RR^{\dnl}$, and $\mb{\xi}^{\star} \in \md{X}^{\star}$,
    $$\nabla^{2} \red{\mc{K}}(\mb{\xi}) (\mb{v}, \mb{v}) \geq \norm{\nabla \red{\mc{R}}(\mb{\xi}) \mb{v}}_{a}^{2} - (\norm{\red{\mc{R}}(\mb{\xi}) - \red{\mc{R}}(\mb{\xi}^{\star})}_{a} + \inf_{v \in V} \norm{u^{\star} - v}_{a}) \norm{\nabla^{2} \red{\mc{R}}(\mb{\xi}) (\mb{v}, \mb{v})}_{a}.$$
\end{lemma}
\begin{proof}
    See \app{convexity}.
\end{proof}

We now show how \lem{convexity} can serve to prove that $\red{\mc{K}}$ is directionally convex in a neighbourhood of its global minimisers. First, by \lem{bounds-best-linear} under \assum{spd} and \subassum{linear-boundedness-regularity}{1}, we infer
\begin{align*}
    \norm{\red{\mc{R}}(\mb{\xi}) - \red{\mc{R}}(\mb{\eta})}_{U} & = \norm{\best{\mb{w}}(\mb{\xi})^{\ast} \mb{\phi}(\mb{\xi}) - \best{\mb{w}}(\mb{\xi})^{\ast} \mb{\phi}(\mb{\eta})}_{U}                                                                     \\
                                                                & \leq \norm{[\best{\mb{w}}(\mb{\xi}) - \best{\mb{w}}(\mb{\eta})]^{\ast} \mb{\phi}(\mb{\xi})}_{U} + \norm{\best{\mb{w}}(\mb{\eta})^{\ast} [\mb{\phi}(\mb{\xi}) - \mb{\phi}(\mb{\eta})]}_{U} \\
                                                                & \leq 2 (1 + \kappa_{\max}) \frac{\norm{\ell}_{U}}{\alpha \omega_{\min}} \norm{\mb{\phi}}_{U, 2, \infty} \norm{\mb{\phi}(\mb{\xi}) - \mb{\phi}(\mb{\eta})}_{U, 2}.
\end{align*}
From the coercivity and boundedness of $a$, we obtain $\norm{\cdot}_{a} \leq \alpha^{-1} \norm{a}_{U \times U} \norm{\cdot}_{U}$ and therefore
$$\nabla^{2} \red{\mc{K}}(\mb{\xi}) (\mb{v}, \mb{v}) \geq \norm{\nabla \red{\mc{R}}(\mb{\xi}) \mb{v}}_{a}^{2} - (C \norm{\mb{\phi}(\mb{\xi}) - \mb{\phi}(\mb{\xi}^{\star})}_{U, 2} + \inf_{v \in V} \norm{u^{\star} - v}_{a}) \norm{\nabla^{2} \red{\mc{R}}(\mb{\xi}) (\mb{v}, \mb{v})}_{a},$$
where $C = 2 \kappa_{\max} (1 + \kappa_{\max}) \norm{\ell}_{U} / \alpha$. We would like this lower bound to be non-negative in the direction $\mb{v} = \mb{\xi}^{\star} - \mb{\xi}$ pointing to a global minimum that is closest to $\mb{\xi}$, that is, for some $\mb{\xi}^{\star} \in \Pi_{\psi}^{\star}(\mb{\xi})$. This condition is equivalent to
$$\norm{\nabla \red{\mc{R}}(\mb{\eta}) (\mb{\xi}^{\star} - \mb{\xi})}_{a}^{2} \geq (C \norm{\mb{\phi}(\mb{\xi}) - \mb{\phi}(\mb{\xi}^{\star})}_{U, 2} + \inf_{v \in V} \norm{u^{\star} - v}_{a}) \norm{\nabla^{2} \red{\mc{R}}(\mb{\eta}) (\mb{\xi}^{\star} - \mb{\xi}, \mb{\xi}^{\star} - \mb{\xi})}_{a},$$
for all $\mb{\eta} \in \cc{\mb{\xi}}{\mb{\xi}^{\star}}$. We are interested in satisfying directional convexity in a neighbourhood of the type $\Lambda_{\psi}^{\star}(\rho)$ for some $\rho > 0$. Assuming that $\mb{\phi}$ is Lipschitz (\subassum{linear-boundedness-regularity}{2}), a stronger condition for directional convexity in $\Lambda_{\psi}^{\star}(\rho)$ is the following: for all $\mb{\xi} \in \Lambda_{\psi}^{\star}(\rho)$, there exists $\mb{\xi}^{\star} \in \Pi_{\psi}^{\star}(\mb{\xi})$, such that for all $\mb{\eta} \in [\mb{\xi}, \mb{\xi}^{\star}]$, it holds
$$\norm{\nabla \red{\mc{R}}(\mb{\eta}) (\mb{\xi}^{\star} - \mb{\xi})}_{a}^{2} \geq (C L_{\mb{\phi}} \rho + \inf_{v \in V} \norm{u^{\star} - v}_{a}) \norm{\nabla^{2} \red{\mc{R}}(\mb{\eta}) (\mb{\xi}^{\star} - \mb{\xi}, \mb{\xi}^{\star} - \mb{\xi})}_{a},$$
where $L_{\mb{\phi}}$ is the Lipschitz constant of $\mb{\phi}$.

Intuitively, this condition characterises a form of quantitative directional convexity in directions pointing towards global minimisers. It ensures that, in a sufficiently small neighbourhood of the global minimisers, the directional derivative of $\red{\mc{K}}$ in the descent direction $\mb{\xi}^{\star} - \mb{\xi}$ dominates the directional curvature, with a bound that scales linearly in the distance to the minimisers. This provides a measure of how sharply the function descends towards its minima, and can be interpreted as a localised, geometric analogue of gradient dominance. This condition has been introduced in~\cite{traonmilin2023basins} in the context of non-convex inverse problems, where it plays a central role in quantifying the size of the basins of attraction around global minimisers.

%% file: 5.conclusion.tex
We develop a general optimisation framework for solving variational \acp{pde} over nonlinear approximation spaces. The method combines energy minimisation in linear parameters with constrained mirror descent for nonlinear parameters. We provide a theoretical foundation ensuring local and global convergence under structural assumptions, including differentiability, boundedness, and directional convexity of the discrete energy functional.

These assumptions are stated in a modular fashion, enabling applicability to a broad class of nonlinear approximation manifolds. Examples include adaptive bases built from piecewise polynomials, wavelets, or radial basis functions. The framework also naturally extends to sparse grid methods, where hierarchical decompositions and selective tensor-product combinations replace full tensor-product spaces \cite{bungartz2004sparse}. In a companion paper \cite{companion}, we explore its application to overlapping tensor-product free-knot B-spline spaces. A detailed study of other constructions is left for future work.

Another interesting direction is the preconditioning of the optimisation process through Hessian-based techniques. In particular, preconditioning the gradient using (an approximation of) the Hessian matrix---as in second-order methods such as Newton or quasi-Newton algorithms, or in natural gradient descent \cite{martens2020new}---could help mitigate ill-conditioning in the nonlinear optimisation problem. This may accelerate convergence and improve robustness by incorporating curvature information, accounting for the geometry of the parameter space, and better balancing the scales of different optimisation variables.

%% file: 6.acknowledgements.tex
This research was partially funded by the Australian Government through the Australian Research Council (project DP220103160). A. Magueresse gratefully acknowledges the Monash Graduate Scholarship from Monash University.

%% file: 7.proofs.tex
\subsection{Proof of \texorpdfstring{\lem{consistency}}{Lemma~\ref{lem:consistency}}}
\label{app:consistency}

\begin{proof}
    Let $\mb{w} \in \operatorname{Ker}(\mb{A}(\mb{\xi}))$. By coercivity of the continuous bilinear form $a$, we have
    $$\alpha \norm{\mc{R}(\mb{w}, \mb{\xi})}_{U}^{2} \leq a(\mc{R}(\mb{w}, \mb{\xi}), \mc{R}(\mb{w}, \mb{\xi})) = \mb{w}^{\ast} \mb{A}(\mb{\xi}) \mb{w} = 0,$$
    and thus $\mc{R}(\mb{w}, \mb{\xi}) = 0$. In particular, this implies
    $$\ell(\mc{R}(\mb{w}, \mb{\xi})) = \mb{w}^{\ast} \mb{\ell}(\mb{\xi}) = 0.$$
    Let us now show that $\mb{\ell}(\mb{\xi}) \in \operatorname{Im}(\mb{A}(\mb{\xi}))$. Since $\mb{A}(\mb{\xi})$ is symmetric with real coefficients, it admits an orthogonal diagonalisation $\mb{A}(\mb{\xi}) = \mb{Q}(\mb{\xi})^{\ast} \mb{\Lambda}(\mb{\xi}) \mb{Q}(\mb{\xi})$, where $\mb{Q}(\mb{\xi}) \in \RR^{\dl \times \dl}$ is orthogonal and $\mb{\Lambda}(\mb{\xi}) \in \RR^{\dl \times \dl}$ is diagonal, with diagonal entries $(\lambda_{k})_{k \in \rg{1}{\dl}}$, the eigenvalues of $\mb{A}(\mb{\xi})$. Let $(\mb{q}_{k})_{k \in \rg{1}{\dl}}$ denote the columns of $\mb{Q}(\mb{\xi})$, so that $\mb{q}_{k}$ is an eigenvector associated with $\lambda_{k}$. In particular, $(\mb{q}_{k})_{k \in \rg{1}{\dl}}$ is a basis so there exist $(c_{k})_{k \in \rg{1}{\dl}}$ such that $\mb{\ell}(\mb{\xi}) = \sum_{k = 1}^{\dl} c_{k} \mb{q}_{k}$. Let $k \in \rg{1}{\dl}$ such that $\lambda_{k} = 0$. Then $\mb{q}_{k} \in \operatorname{Ker}(\mb{A}(\mb{\xi}))$ and by what is above, we infer that $\mb{q}_{k}^{\ast} \mb{\ell}(\mb{\xi}) = c_{k} = 0$. This shows that $\mb{\ell}(\mb{\xi}) \in \operatorname{Span}(\{\mb{q}_{k} : \lambda_{k} \neq 0\}) = \operatorname{Im}(\mb{A}(\mb{\xi}))$, so the linear system is consistent.

    If $\mb{w}_{1}, \mb{w}_{2} \in \md{W}$ are two solutions of the system, then $\mb{w}_{1} - \mb{w}_{2} \in \operatorname{Ker}(\mb{A}(\mb{\xi}))$, and therefore $\mc{R}(\mb{w}_{1}, \mb{\xi}) = \mc{R}(\mb{w}_{2}, \mb{\xi})$. This shows that the realisation is independent of the choice of solution.
\end{proof}

\subsection{Proof of \texorpdfstring{\lem{energy-decrease}}{Lemma~\ref{lem:energy-decrease}}}
\label{app:energy-decrease}

\begin{proof}
    Let $\mb{w} \in \md{W}$ and $\mb{\xi} \in \md{X}$, and $\mb{w}_{+}$ correspond to the full linear solve. Note that $\DW \mc{K}(\mb{w}, \mb{\xi}) = \mb{A}(\mb{\xi}) \mb{w} - \mb{\ell} = \mb{A}(\mb{\xi}) (\mb{w} - \mb{w}_{+})$. We thus express
    \begin{align*}
        \mc{K}(\mb{w}_{+}, \mb{\xi}) & = \mc{K}(\mb{w}, \mb{\xi}) - \frac{1}{2} (\mb{w} - \mb{w}_{+})^{\ast} \mb{A}(\mb{\xi}) (\mb{w} - \mb{w}_{+})                    \\
                                     & = \mc{K}(\mb{w}, \mb{\xi}) - \frac{1}{2} \DW \mc{K}(\mb{w}, \mb{\xi})^{\ast} \mb{A}(\mb{\xi})^{-1} \DW \mc{K}(\mb{w}, \mb{\xi}) \\
                                     & \leq \mc{K}(\mb{w}, \mb{\xi}) - \frac{1}{2} \lambda_{\max}(\mb{A}(\mb{\xi}))^{-1} \norm{\DW \mc{K}(\mb{w}, \mb{\xi})}_{2, \ast}
    \end{align*}
    If $\mb{w}_{+}$ corresponds to the steepest descent update, we find
    \begin{align*}
        \mc{K}(\mb{w}_{+}, \mb{\xi}) & = \mc{K}(\mb{w}, \mb{\xi}) + \frac{1}{2} \beta^{2} \mb{r}^{\ast} \mb{A}(\mb{\xi}) \mb{r} + \beta \mb{r}^{\ast} \mb{A}(\mb{\xi}) \mb{w} - \beta \mb{r}^{\ast} \mb{\ell}(\mb{\xi}) \\
                                     & = \mc{K}(\mb{w}, \mb{\xi}) - \frac{1}{2} \frac{(\mb{r}, \mb{r})_{2}^{2}}{(\mb{r}, \mb{A}(\mb{\xi}) \mb{r})_{2}}.
    \end{align*}
    Here again we conclude using the largest eigenvalue of $\mb{A}(\mb{\xi})$ and the fact that $\DW \mc{K}(\mb{w}, \mb{\xi}) = \mb{r}$.

    To bound the largest eigenvalue of $\mb{A}(\mb{\xi})$, the continuity of $a$ gives
    $$\mb{w}^{\ast} \mb{A}(\mb{\xi}) \mb{w} = a(\mc{R}(\mb{w}, \mb{\xi}), \mc{R}(\mb{w}, \mb{\xi})) \leq \norm{a}_{U \times U} \norm{\mc{R}(\mb{w}, \mb{\xi})}_{U}^{2}.$$
    Now, the Gram matrix $\mb{G}(\mb{\xi})$ allows one to write
    $$\norm{\mc{R}(\mb{w}, \mb{\xi})}_{U}^{2} = \mb{w}^{\ast} \mb{G}(\mb{\xi}) \mb{w} \leq \lambda_{\max}(\mb{G}(\mb{\xi})) \norm{\mb{w}}_{2}^{2}.$$
    Since the Frobenius norm of a symmetric matrix is equal to the sum the squares of its eigenvalues, one can bound $\lambda_{\max}(\mb{G}(\mb{\xi}))^{2}$ by the Frobenius norm of $\mb{G}(\mb{\xi})$. Using the Cauchy-Schwarz inequality, we obtain
    $$\lambda_{\max}(\mb{G}(\mb{\xi}))^{2} \leq \sum_{i, j = 1}^{\dl} (\mb{\phi}_{i}(\mb{\xi}), \mb{\phi}_{j}(\mb{\xi}))_{U}^{2} \leq \sum_{i, j = 1}^{\dl} \norm{\mb{\phi}_{i}(\mb{\xi})}_{U}^{2} \norm{\mb{\phi}_{j}(\mb{\xi})}_{U}^{2} = \norm{\mb{\phi}(\mb{\xi})}_{U, 2}^{4}.$$
    We conclude that $\lambda_{\max}(\mb{G}(\mb{\xi})) \leq \norm{\mb{\phi}(\mb{\xi})}_{U, 2}^{2}$ and $\lambda_{\max}(\mb{A}(\mb{\xi})) \leq \norm{a}_{U \times U} \norm{\mb{\phi}(\mb{\xi})}_{U, 2}^{2}$.
\end{proof}

\subsection{Tools for \texorpdfstring{\lem{local-convergence}}{Lemma~\ref{lem:local-convergence}}, \texorpdfstring{\lem{surrogate}}{Lemma~\ref{lem:surrogate}} and \texorpdfstring{\lem{global-convergence}}{Lemma~\ref{lem:global-convergence}}}
Let $\mb{w} \in \md{W}$. The function $\mc{K}(\mb{w}, \cdot)$ has Hölder gradients with exponent $\nu$ and constant $L(\mb{w})$. It is well-known that the fundamental theorem of calculus then yields
$$\mc{K}(\mb{w}, \mb{\eta}) \leq \mc{K}(\mb{w}, \mb{\xi}) + \inner{\DX \mc{K}(\mb{w}, \mb{\xi})}{\mb{\eta} - \mb{\xi}}_{\md{X}} + \frac{L(\mb{w})}{1 + \nu} \norm{\mb{\eta} - \mb{\xi}}_{\md{X}}^{1 + \nu}$$
for all $\mb{\xi}, \mb{\eta} \in \md{X}$. We now repeat an argument used in \cite{nesterov2015universal} to convert Hölder continuity into Lipschitz continuity up to an arbitrary small constant: by Young's inequality, it holds
$$a b \leq \frac{a^{p}}{p} + \frac{b^{q}}{q},$$
for all $a, b \geq 0$ and $p, q \geq 1$ satisfying $1/p + 1/q = 1$. Taking $a = \norm{\mb{\xi} - \mb{\eta}}_{\md{X}}^{1 + \nu}$, $p = 2 / (1 + \nu)$ and therefore $q = 2 / (1 - \nu)$, we find
$$\frac{L(\mb{w})}{1 + \nu} \norm{\mb{\xi} - \mb{\eta}}_{\md{X}}^{1 + \nu} \leq \frac{L(\mb{w})}{2 b} \norm{\mb{\xi} - \mb{\eta}}_{\md{X}}^{2} + \frac{1}{2} \frac{1 - \nu}{1 + \nu} b^{\frac{1 + \nu}{1 - \nu}} L(\mb{w}).$$
Letting $\epsilon = \frac{1}{2} \frac{1 - \nu}{1 + \nu} b^{\frac{1 + \nu}{1 - \nu}} L(\mb{w}) > 0$ and solving for $b$ in terms of $\epsilon$, we obtain
$$\mc{K}(\mb{w}, \mb{\eta}) \leq \mc{K}(\mb{w}, \mb{\xi}) + \inner{\DX \mc{K}(\mb{w}, \mb{\xi})}{\mb{\eta} - \mb{\xi}}_{\md{X}} + \frac{1}{2} L_{\nu, \epsilon}(\mb{w}) \norm{\mb{\eta} - \mb{\xi}}_{\md{X}}^{1 + \nu} + \epsilon$$
where we have introduced the Lipschitz constant
$$L_{\nu, \epsilon}(\mb{w}) \doteq \left(\frac{1}{2 \epsilon} \frac{1 - \nu}{1 + \nu}\right)^{\frac{1 - \nu}{1 + \nu}} L(\mb{w})^{\frac{2}{1 + \nu}}.$$
If $\nu = 1$, one can also take $\epsilon = 0$, and in that case $L_{\nu, \epsilon}(\mb{w}) = L(\mb{w})$.

\subsection{Proof of \texorpdfstring{\lem{local-convergence}}{Lemma~\ref{lem:local-convergence}}}
\label{app:local-convergence}

\begin{proof}
    Let $(\mb{w}, \mb{\xi}) \in \md{W} \times \md{X}$. The Hölder continuity of $\DX \mc{K}$ at $\mb{w}$ implies
    $$\mc{K}(\mb{w}, \mb{\xi}_{+}) \leq \mc{K}(\mb{w}, \mb{\xi}) + \inner{\DX \mc{K}(\mb{w}, \mb{\xi})}{\mb{\xi}_{+} - \mb{\xi}}_{\md{X}} + \frac{1}{2} L_{\nu, \epsilon}(\mb{w}) \norm{\mb{\xi}_{+} - \mb{\xi}}_{\md{X}}^{2} + \epsilon,$$
    for any $\epsilon > 0$. The first-order optimality condition characterising $\mb{\xi}_{+}$ \eqref{eq:def-prox} states
    $$-\gamma \DX \mc{K}(\mb{w}, \mb{\xi}) - \nabla \psi(\mb{\xi}_{+}) + \nabla \psi(\mb{\xi}) \in N_{\md{X}}(\mb{\xi}_{+}).$$
    In particular,
    $$\inner{-\gamma \DX \mc{K}(\mb{w}, \mb{\xi}) - \nabla \psi(\mb{\xi}_{+}) + \nabla \psi(\mb{\xi})}{\mb{\xi} - \mb{\xi}_{+}}_{\md{X}} \leq 0,$$
    which is equivalent to
    $$\gamma \inner{\DX \mc{K}(\mb{w}, \mb{\xi})}{\mb{\xi}_{+} - \mb{\xi}}_{\md{X}} \leq -\inner{\nabla \psi(\mb{\xi}_{+}) - \nabla \psi(\mb{\xi})}{\mb{\xi}_{+} - \mb{\xi}}_{\md{X}}.$$
    The strong convexity of $\psi$ gives the upper bound $- \mu \norm{\mb{\xi}_{+} - \mb{\xi}}_{\md{X}}^{2}$ for the right-hand side. Combining these inequalities and the definition of the gradient mapping, we reach
    $$\mc{K}(\mb{w}, \mb{\xi}_{+}) \leq \mc{K}(\mb{w}, \mb{\xi}) - \frac{1}{2} \gamma (2 \mu - \gamma L_{\nu, \epsilon}(\mb{w})) \norm{\gradmap{\mb{w}}{\mb{\xi}}{\gamma}}_{\md{X}, \ast}^{2} + \epsilon.$$
    Using the energy decrease condition \eqref{eq:decrease}, we arrive at
    \begin{equation}
        \label{eq:local-descent}
        \mc{K}(\mb{w}_{+}, \mb{\xi}_{+}) \leq \mc{K}(\mb{w}, \mb{\xi}) - \frac{1}{2} \gamma (2 \mu - \gamma L_{\nu, \epsilon}(\mb{w})) \norm{\gradmap{\mb{w}}{\mb{\xi}}{\gamma}}_{\md{X}, \ast}^{2} - \frac{1}{2} \beta(\mb{\xi})^{-1} \norm{\DW \mc{K}(\mb{w}, \mb{\xi}_{+})}_{\md{W}, \ast}^{2} + \epsilon.
    \end{equation}
    Let now $(\mb{w}_{0}, \mb{\xi}_{0}) \in \md{W} \times \md{X}$ and $(\mb{w}_{k}, \mb{\xi}_{k})_{k \geq 0}$ denote the iterates produced by \alg{optimiser} with step sizes $\gamma_{k} > 0$. The estimate \eqref{eq:local-descent} holds at each iteration. Summing this inequality and recognising a telescoping sum, we obtain
    $$\sum_{k = 0}^{n-1} A_{k} \norm{\gradmap{\mb{w}_{k}}{\mb{\xi}_{k}}{\gamma_{k}}}_{\md{X}, \ast}^{2} + B_{k} \norm{\DW \mc{K}(\mb{w}_{k}, \mb{\xi}_{k+1})}_{\md{W}, \ast}^{2} \leq 2 (\mc{K}(\mb{w}_{0}, \mb{\xi}_{0}) - \mc{K}(\mb{w}_{n}, \mb{\xi}_{n}) + n \epsilon),$$
    where $A_{k} \doteq \gamma_{k} (2 \mu - \gamma_{k} L_{\nu, \epsilon}(\mb{w}_{k}))$ and $B_{k} \doteq \beta(\mb{\xi}_{k})^{-1}$
    Assuming $\sum_{k = 0}^{n-1} \min(A_{k}, B_{k}) > 0$ and dividing by that sum to form a weighted mean, we conclude
    $$\min_{0 \leq k \leq n-1} \left\{\norm{\gradmap{\mb{w}_{k}}{\mb{\xi}_{k}}{\gamma_{k}}}_{\md{X}, \ast}^{2} + \norm{\DW \mc{K}(\mb{w}_{k}, \mb{\xi}_{k+1})}_{\md{W}, \ast}^{2}\right\} \leq \frac{2 (\mc{K}(\mb{w}_{0}, \mb{\xi}_{0}) - \mc{K}^{\star} + n \epsilon)}{\sum_{k = 0}^{n-1} \min(A_{k}, B_{k})},$$
    where we also used $\mc{K}(\mb{w}_{n}, \mb{\xi}_{n}) \geq \mc{K}^{\star}$.
\end{proof}

\subsection{Proof of \texorpdfstring{\lem{surrogate}}{Lemma~\ref{lem:surrogate}}}
\label{app:surrogate}

\begin{proof}
    The optimality condition defining $\mb{\xi}_{+} \in \proxmap{\mb{w}}{\mb{\xi}}{\gamma}$ \eqref{eq:def-prox} states $\nabla \psi(\mb{\xi}) - \nabla \psi(\mb{\xi}_{+}) - \gamma \DX \mc{K}(\mb{w}, \mb{\xi}) \in N_{\md{X}}(\mb{\xi}_{+})$. Dividing by $\gamma$ and adding and subtracting $\DX \mc{K}(\mb{w}, \mb{\xi}_{+})$, this is equivalent to
    $$- \DX \mc{K}(\mb{w}, \mb{\xi}_{+}) + \underbrace{[\DX \mc{K}(\mb{w}, \mb{\xi}_{+}) - \DX \mc{K}(\mb{w}, \mb{\xi}) + \gamma^{-1} (\nabla \psi(\mb{\xi}) - \nabla \psi(\mb{\xi}_{+}))]}_{\doteq \mb{v}} \in N_{\md{X}}(\mb{\xi}_{+}).$$
    Using the Hölder continuity of $\DX \mc{K}$, we compute
    \begin{align*}
        \norm{\mb{v}}_{\md{X}, \ast} & \leq \norm{\DX \mc{K}(\mb{w}, \mb{\xi}_{+}) - \DX \mc{K}(\mb{w}, \mb{\xi})}_{\md{X}, \ast} + \gamma^{-1} \norm{\nabla \psi(\mb{\xi}_{+}) - \nabla \psi(\mb{\xi})}_{\md{X}, \ast} \\
                                     & \leq L \norm{\mb{\xi}_{+} - \mb{\xi}}_{\md{X}}^{\nu} + \mu \gamma^{-1} \norm{\mb{\xi}_{+} - \mb{\xi}}_{\md{X}, \ast}.
    \end{align*}
    This shows $-\DX \mc{K}(\mb{w}, \mb{\xi}_{+}) \in N_{\md{X}}(\mb{\xi}_{+}) + \overline{B}_{\md{X}, \ast}(0, L c^{\nu} + \mu \gamma^{-1} c)$, where $c = \norm{\mb{\xi}_{+} - \mb{\xi}}_{\md{X}}$.

    Let now $h^{2} = \norm{\gradmap{\mb{w}}{\mb{\xi}}{\gamma}}_{\md{X}, \ast}^{2} + \norm{\DW \mc{K}(\mb{w}, \mb{\xi}_{+})}_{\md{W}, \ast}^{2}$. In particular, the two terms are smaller than $h$. Therefore $\norm{\mb{\xi}_{+} - \mb{\xi}} \leq \gamma h$ and thus $-\DX \mc{K}(\mb{w}, \mb{\xi}_{+}) \in N_{\md{X}}(\mb{\xi}_{+}) + \overline{B}_{\md{X}, \ast}(0, L (\gamma h)^{\nu} + \mu h)$. We conclude that $(\mb{w}, \mb{\xi}_{+})$ is a $(L (\gamma h)^{\nu} + \mu h)$-quasi-stationary point.
\end{proof}

\subsection{Proof of \texorpdfstring{\lem{global-convergence}}{Lemma~\ref{lem:global-convergence}}}
\label{app:global-convergence}

\begin{proof}
    Let $\mb{\xi} \in \Lambda_{\psi}^{\star}(\rho)$ and $\mb{\xi}_{+} \in \proxmap{\best{\mb{w}}(\mb{\xi})}{\mb{\xi}}{\gamma}$ for some $\gamma > 0$. By the directional convexity assumption, there exists $\mb{\xi}^{\star} \in \Pi_{\psi}^{\star}(\mb{\xi})$ such that $\red{\mc{K}}$ is convex on the segment $\cc{\mb{\xi}}{\mb{\xi}^{\star}}$. Therefore
    \begin{align}
        \label{eq:convexity-lipschitz}
        \red{\mc{K}}(\mb{\xi}) & \leq \red{\mc{K}}(\mb{\xi}^{\star}) + \inner{\nabla \red{\mc{K}}(\mb{\xi})}{\mb{\xi} - \mb{\xi}^{\star}}_{\md{X}}                                                               \nonumber                                      \\
                               & = \mc{K}^{\star} + \inner{\nabla \red{\mc{K}}(\mb{\xi})}{\mb{\xi} - \mb{\xi}_{+}}_{\md{X}} + \inner{\nabla \red{\mc{K}}(\mb{\xi})}{\mb{\xi}_{+} - \mb{\xi}^{\star}}_{\md{X}}                     \nonumber                     \\
                               & \leq \mc{K}^{\star} + \red{\mc{K}}(\mb{\xi}) - \red{\mc{K}}(\mb{\xi}_{+}) + \frac{1}{2} \red{L} \norm{\mb{\xi}_{+} - \mb{\xi}}_{\md{X}}^{2} + \inner{\nabla \red{\mc{K}}(\mb{\xi})}{\mb{\xi}_{+} - \mb{\xi}^{\star}}_{\md{X}},
    \end{align}
    where we used the Lipschitz continuity of $\nabla \red{\mc{K}}$ to bound $\inner{\nabla \red{\mc{K}}(\mb{\xi})}{\mb{\xi} - \mb{\xi}_{+}}_{\md{X}}$. The first-order optimality condition \eqref{eq:def-prox} defining $\mb{\xi}_{+}$ at $\mb{\xi}^{\star}$ yields
    $$\inner{-\gamma \nabla \red{\mc{K}}(\mb{\xi}) - \nabla \psi(\mb{\xi}_{+}) + \nabla \psi(\mb{\xi}) \in N_{\md{X}}(\mb{\xi}_{+})}{\mb{\xi}^{\star} - \mb{\xi}_{+}}_{\md{X}} \leq 0,$$
    which is equivalent to
    $$\gamma \inner{\nabla \red{\mc{K}}(\mb{\xi})}{\mb{\xi}_{+} - \mb{\xi}^{\star}}_{\md{X}} \leq \inner{\nabla \psi(\mb{\xi}) - \nabla \psi(\mb{\xi}_{+})}{\mb{\xi}_{+} - \mb{\xi}^{\star}}_{\md{X}}.$$
    From the definition of the Bregman divergence, we rewrite
    $$\inner{\nabla \psi(\mb{\xi}) - \nabla \psi(\mb{\xi}_{+})}{\mb{\xi}_{+} - \mb{\xi}^{\star}}_{\md{X}} = D_{\psi}(\mb{\xi}^{\star}; \mb{\xi}) - D_{\psi}(\mb{\xi}^{\star}; \mb{\xi}_{+}) - D_{\psi}(\mb{\xi}_{+}; \mb{\xi}).$$
    Plugging this inequality in \eqref{eq:convexity-lipschitz} and simplifying, we reach
    $$\red{\mc{K}}(\mb{\xi}_{+}) \leq \mc{K}^{\star} - \gamma^{-1} D_{\psi}(\mb{\xi}_{+}; \mb{\xi}) + \frac{1}{2} \red{L} \norm{\mb{\xi}_{+} - \mb{\xi}}_{\md{X}}^{2} + \gamma^{-1} (D_{\psi}(\mb{\xi}^{\star}; \mb{\xi}) - D_{\psi}(\mb{\xi}^{\star}; \mb{\xi}_{+})).$$
    Now, using the fact that $\delta_{\psi}^{\star}(\mb{\xi}) = D_{\psi}(\mb{\xi}^{\star}; \mb{\xi})$, the inequality $\delta_{\psi}^{\star}(\mb{\xi}_{+}) = D_{\psi}(\mb{\xi}^{\star}; \mb{\xi}_{+})$ and the strong convexity inequality $D_{\psi}(\mb{\xi}; \mb{\eta}) \geq \mu \norm{\mb{\xi} - \mb{\eta}}_{\md{X}}^{2} / 2$, we finally obtain
    \begin{equation}
        \label{eq:global-descent}
        \red{\mc{K}}(\mb{\xi}_{+}) \leq \mc{K}^{\star} - \frac{1}{2} (\mu \gamma^{-1} - \red{L}) \norm{\mb{\xi}_{+} - \mb{\xi}}_{\md{X}}^{2} + \gamma^{-1} (\delta_{\psi}^{\star}(\mb{\xi}) - \delta_{\psi}^{\star}(\mb{\xi}_{+})).
    \end{equation}
    Assuming $\gamma \red{L} \leq \mu$, we find that $\delta_{\psi}^{\star}(\mb{\xi}_{+}) \leq \delta_{\psi}^{\star}(\mb{\xi}) \leq \rho$, and therefore $\mb{\xi}_{+} \in \Lambda_{\psi}^{\star}(\rho)$.

    A quick induction shows that if $\mb{\xi}_{0} \in \Lambda_{\psi}^{\star}(\rho)$ and $\gamma_{k} \red{L} \leq \mu$ for all $k \geq 0$, then the iterates of \alg{optimiser} remain in $\Lambda_{\psi}^{\star}(\rho)$, and that the estimate \eqref{eq:global-descent} holds at each step. According to \lem{local-convergence}, the condition $\gamma_{k} \red{L} \leq \mu$ also ensures that $(\red{\mc{K}}(\mb{\xi}_{k}))_{k \geq 0}$ is non-increasing. We apply \eqref{eq:global-descent} at each step and multiply it by $\gamma_{k}$ in view of telescoping $\delta_{\psi}^{\star}(\mb{\xi}_{k})$. We recognise a weighted average of $(\red{\mc{K}}(\mb{\xi}_{k}) - \mc{K}^{\ast})$ and find
    $$\red{\mc{K}}(\mb{\xi}_{n}) - \mc{K}^{\star} \leq \frac{\sum_{k = 0}^{n-1} \gamma_{k} (\red{\mc{K}}(\mb{\xi}_{k+1}) - \mc{K}^{\star})}{\sum_{k = 0}^{n-1} \gamma_{k}} \leq \frac{\delta_{\psi}^{\star}(\mb{\xi}_{0}) - \delta_{\psi}^{\star}(\mb{\xi}_{n})}{\sum_{k = 0}^{n-1} \gamma_{k}} \leq \frac{\delta_{\psi}^{\star}(\mb{\xi}_{0})}{\sum_{k = 0}^{n-1} \gamma_{k}},$$
    where the first inequality comes from the monotonicity of $(\red{\mc{K}}(\mb{\xi}_{k}))_{k \geq 0}$.
\end{proof}

\subsection{Proof of \texorpdfstring{\corol{nonlinear-cea}}{Corollary~\ref{corol:nonlinear-cea}}}
\label{app:nonlinear-cea}

\begin{proof}
    Since $\red{\mc{K}}$ satisfies the assumptions of \lem{global-convergence}, for all $n \geq 1$ and $v \in V$, we have
    \begin{align*}
        \frac{1}{2} \norm{\red{\mc{R}}(\mb{\xi}_{n}) - u^{\star}}_{a}^{2} & = \red{\mc{K}}(\mb{\xi}_{n}) - \mc{J}(u^{\star})                                                                 \\
                                                                          & = \mc{J}(v) - \mc{J}(u^{\star}) + \red{\mc{K}}(\mb{\xi}_{n}) - \mc{J}(v)                                         \\
                                                                          & \leq \frac{1}{2} \norm{v - u^{\star}}_{a}^{2} + \red{\mc{K}}(\mb{\xi}_{n}) - \mc{K}^{\star}                      \\
                                                                          & \leq \frac{1}{2} \norm{v - u^{\star}}_{a}^{2} + \frac{\red{L} \delta_{\psi}^{\star}(\mb{\xi}_{0})}{\zeta \mu n}.
    \end{align*}
    Taking the infimum on $v \in V$ yields the result of the corollary.
\end{proof}

\subsection{Proof of \texorpdfstring{\lem{regularity-energy}}{Lemma~\ref{lem:regularity-energy}}}
\label{app:regularity-energy}

\begin{proof}
    Let $\mb{v}, \mb{w} \in \md{W}$, and $\mb{\xi}, \mb{\eta} \in \md{X}$. Let $i \in \rg{1}{\dl}$ and $\partial_{i}$ denote the derivative with respect to the $i$-th linear parameter. We express
    $$\partial_{i} \mc{K}(\mb{v}, \mb{\xi}) = a(\mb{v}^{\ast} \mb{\phi}(\mb{\xi}), \mb{\phi}_{i}(\mb{\xi})) - \ell(\mb{\phi}_{i}(\mb{\xi})).$$
    Using the fact that $x = c + d$ and $y = c - d$, where $c = (x + y) / 2$ and $d = (x - y) / 2$, to symmetrise $\partial_{i} \mc{K}(\mb{v}, \mb{\xi}) - \partial_{i} \mc{K}(\mb{w}, \mb{\eta})$, we find
    \begin{align*}
        \abs{\partial_{i} \mc{K}(\mb{v}, \mb{\xi}) - \partial_{i} \mc{K}(\mb{w}, \mb{\eta})} & \leq \frac{1}{4} \abs{a([\mb{v} - \mb{w}]^{\ast} [\mb{\phi}(\mb{\xi}) + \mb{\phi}(\mb{\eta})], \mb{\phi}_{i}(\mb{\xi}) + \mb{\phi}_{i}(\mb{\eta}))} \\
                                                                                             & + \frac{1}{4} \abs{a([\mb{v} + \mb{w}]^{\ast} [\mb{\phi}(\mb{\xi}) - \mb{\phi}(\mb{\eta})], \mb{\phi}_{i}(\mb{\xi}) + \mb{\phi}_{i}(\mb{\eta}))}    \\
                                                                                             & + \frac{1}{2} \abs{a(\mb{v}^{\ast} \mb{\phi}(\mb{\xi}) + \mb{w}^{\ast} \mb{\phi}(\mb{\eta}), \mb{\phi}_{i}(\mb{\xi}) - \mb{\phi}_{i}(\mb{\eta}))}   \\
                                                                                             & + \abs{\ell(\mb{\phi}_{i}(\mb{\xi}) - \mb{\phi}_{i}(\mb{\eta}))}.
    \end{align*}
    The inequality of the lemma follows from the boundedness of $a$ and $\ell$, the triangle inequality, the Cauchy-Schwarz inequality $\norm{\mb{v}^{\ast} \mb{\phi}(\mb{\xi})}_{U} \leq \norm{\mb{v}}_{2} \norm{\mb{\phi}(\mb{\xi})}_{U, 2}$, the fact that $x + y \leq 2 \max(x, y)$ and $\max(x y, z t) \leq \max(x, z) \max(y, t)$ for all $x, y, z, t \geq 0$.

    For all $i \in \rg{1}{\dnl}$, let $\partial_{i}$ now denote the derivative with respect to the $i$-th nonlinear parameter. Since $\partial_{i} \mb{\phi}_{k}(\mb{\xi}) \in U$, we have
    $$\partial_{i} \mc{K}(\mb{v}, \mb{\xi}) = a(\mb{v}^{\ast} \mb{\phi}(\mb{\xi}), \mb{v}^{\ast} \partial_{i} \mb{\phi}(\mb{\xi})) - \ell(\mb{v}^{\ast} \partial_{i} \mb{\phi}(\mb{\xi})).$$
    Applying the same symmetrisation technique as above twice to isolate each term with a single subtraction, we find
    \begin{align*}
        \abs{\partial_{i} \mc{K}(\mb{v}, \mb{\xi}) - \partial_{i} \mc{K}(\mb{w}, \mb{\eta})} & \leq \frac{1}{4} \abs{a(\mb{v}^{\ast} \mb{\phi}(\mb{\xi}) + \mb{w}^{\ast} \mb{\phi}(\mb{\eta}), [\mb{v} - \mb{w}]^{\ast} [\partial_{i} \mb{\phi}(\mb{\xi}) + \partial_{i} \mb{\phi}(\mb{\eta})])} \\
                                                                                             & + \frac{1}{4} \abs{a(\mb{v}^{\ast} \mb{\phi}(\mb{\xi}) + \mb{w}^{\ast} \mb{\phi}(\mb{\eta}), [\mb{v} + \mb{w}]^{\ast} [\partial_{i} \mb{\phi}(\mb{\xi}) - \partial_{i} \mb{\phi}(\mb{\eta})])}    \\
                                                                                             & + \frac{1}{4} \abs{a([\mb{v} - \mb{w}]^{\ast} [\mb{\phi}(\mb{\xi}) + \mb{\phi}(\mb{\eta})], \mb{v}^{\ast} \partial_{i} \mb{\phi}(\mb{\xi}) + \mb{w}^{\ast} \partial_{i} \mb{\phi}(\mb{\eta}))}    \\
                                                                                             & + \frac{1}{4} \abs{a([\mb{v} + \mb{w}]^{\ast} [\mb{\phi}(\mb{\xi}) - \mb{\phi}(\mb{\eta})], \mb{v}^{\ast} \partial_{i} \mb{\phi}(\mb{\xi}) + \mb{w}^{\ast} \partial_{i} \mb{\phi}(\mb{\eta}))}    \\
                                                                                             & + \frac{1}{2} \abs{\ell([\mb{v} - \mb{w}]^{\ast} [\partial_{i} \mb{\phi}(\mb{\xi}) + \partial_{i} \mb{\phi}(\mb{\eta})])}                                                                         \\
                                                                                             & + \frac{1}{2} \abs{\ell([\mb{v} + \mb{w}]^{\ast} [\partial_{i} \mb{\phi}(\mb{\xi}) - \partial_{i} \mb{\phi}(\mb{\eta})])}.
    \end{align*}
    We conclude using the same arguments as above.
\end{proof}

\subsection{Proof of \texorpdfstring{\lem{bounds-best-linear}}{Lemma~\ref{lem:bounds-best-linear}}}
\label{app:bounds-best-linear}

\begin{proof}
    Let $\mb{\xi} \in \md{X}$. The definition of $\best{\mb{w}}(\mb{\xi})$, \assum{spd} and the Cauchy-Schwarz inequality yield
    \begin{align*}
        \alpha \omega_{\min} \norm{\best{\mb{w}}(\mb{\xi})}_{2}^{2} & \leq \best{\mb{w}}(\mb{\xi})^{\ast} \mb{A}(\mb{\xi}) \best{\mb{w}}(\mb{\xi})               \\
                                                                    & = \best{\mb{w}}(\mb{\xi})^{\ast} \mb{\ell}(\mb{\xi})                                       \\
                                                                    & = \ell(\best{\mb{w}}(\mb{\xi})^{\ast} \mb{\phi}(\mb{\xi}))                                 \\
                                                                    & \leq \norm{\ell}_{U} \norm{\best{\mb{w}}(\mb{\xi})}_{2} \norm{\mb{\phi}(\mb{\xi})}_{U, 2}.
    \end{align*}
    The first inequality of the lemma is obtained by dividing this inequality by $\alpha \omega_{\min} \norm{\best{\mb{w}}(\mb{\xi})}_{2}$ and invoking \subassum{linear-boundedness-regularity}{1} to bound $\norm{\mb{\phi}(\mb{\xi})}_{U, 2}$.

    Let now $\mb{\eta} \in \md{X}$ such that $\omega(\mb{\eta}) > 0$. We first show that the matrix $(1 - t) \mb{A}(\mb{\xi}) + t \mb{A}(\mb{\eta})$ is invertible for all $t \in \cc{0}{1}$. Let $\mb{v} \in \md{W}$. We compute
    \begin{align*}
        \mb{v}^{\ast} [(1 - t) \mb{A}(\mb{\xi}) + t \mb{A}(\mb{\eta})] \mb{v} & = (1 - t) \mb{v}^{\ast} \mb{A}(\mb{\xi}) \mb{v} + t \mb{v}^{\ast} \mb{A}(\mb{\eta}) \mb{v}             \\
                                                                              & \geq (1 - t) \alpha \omega_{\min} \norm{\mb{v}}_{2}^{2} + t \alpha \omega_{\min} \norm{\mb{v}}_{2}^{2} \\
                                                                              & = \alpha \omega_{\min} \norm{\mb{v}}_{2}^{2}.
    \end{align*}
    This shows that the smallest eigenvalue of $(1 - t) \mb{A}(\mb{\xi}) + t \mb{A}(\mb{\eta})$ is greater than $\alpha \omega_{\min}$. This enables the definition of the map
    $$h: \cc{0}{1} \to \md{W}, \quad t \mapsto [(1 - t) \mb{A}(\mb{\xi}) + t \mb{A}(\mb{\eta})]^{-1} [(1 - t) \mb{\ell}(\mb{\xi}) + t \mb{\ell}(\mb{\eta})],$$
    such that $\best{\mb{w}}(\mb{\xi}) - \best{\mb{w}}(\mb{\eta}) = h(0) - h(1)$. This map is differentiable and its derivative is
    \begin{align*}
        h'(t) & = [(1 - t) \mb{A}(\mb{\xi}) + t \mb{A}(\mb{\eta})]^{-1} [\mb{A}(\mb{\xi}) - \mb{A}(\mb{\eta})] [(1 - t) \mb{A}(\mb{\xi}) + t \mb{A}(\mb{\eta})]^{-1} [(1 - t) \mb{\ell}(\mb{\xi}) + t \mb{\ell}(\mb{\eta})] \\
              & + [(1 - t) \mb{A}(\mb{\xi}) + t \mb{A}(\mb{\eta})]^{-1} [\mb{\ell}(\mb{\eta}) - \mb{\ell}(\mb{\xi})].
    \end{align*}
    By the mean value theorem there exists $t \in \oo{0}{1}$ such that $h(1) - h(0) = h'(t)$. Taking the norm of that equality, we infer
    \begin{align*}
        \norm{\best{\mb{w}}(\mb{\xi}) - \best{\mb{w}}(\mb{\eta})}_{2} & \leq \sigma_{\max}([(1 - t) \mb{A}(\mb{\xi}) + t \mb{A}(\mb{\eta})]^{-1})^{2} \sigma_{\max}(\mb{A}(\mb{\xi}) - \mb{A}(\mb{\eta})) \norm{(1 - t) \mb{\ell}(\mb{\xi}) + t \mb{\ell}(\mb{\eta})}_{2} \\
                                                                      & + \sigma_{\max}([(1 - t) \mb{A}(\mb{\xi}) + t \mb{A}(\mb{\eta})]^{-1}) \norm{\ell(\mb{\phi}(\mb{\eta}) - \mb{\phi}(\mb{\xi}))}_{2}                                                                \\
                                                                      & \leq \alpha^{-2} \omega_{\min}^{-2} \norm{\ell}_{U} M_{\mb{\phi}} \sigma_{\max}(\mb{A}(\mb{\xi}) - \mb{A}(\mb{\eta}))                                                                             \\
                                                                      & + \alpha^{-1} \omega_{\min}^{-1} \norm{\ell}_{U} \norm{\mb{\phi}(\mb{\xi}) - \mb{\phi}(\mb{\eta})}_{U, 2},
    \end{align*}
    where we used \subassum{linear-boundedness-regularity}{1} to bound $\max(\norm{\mb{\phi}(\mb{\xi})}_{U, 2}, \norm{\mb{\phi}(\mb{\eta})}_{U, 2})$. We also used the fact that $\norm{\mb{M} \mb{x}}_{2} \leq \sigma_{\max}(\mb{M}) \norm{\mb{x}}_{2}$ for all $\mb{M} \in \RR^{\dl \times \dl}$ and $\mb{x} \in \RR^{\dl}$. Here $\sigma_{\max}(\mb{M})$ denotes the largest singular value of $\mb{M}$. Since the largest singular value of a matrix is bounded by its Frobenius norm, we obtain
    \begin{align*}
        \sigma_{\max}(\mb{A}(\mb{\xi}) - \mb{A}(\mb{\eta}))^{2} & \leq \sum_{i, j = 1}^{\dl} (\mb{A}(\mb{\xi})_{ij} - \mb{A}(\mb{\eta})_{ij})^{2}                                                                                                                                                                                              \\
                                                                & = \sum_{i, j = 1}^{\dl} [\frac{1}{2} a(\mb{\phi}_{i}(\mb{\xi}) - \mb{\phi}_{i}(\mb{\eta}), \mb{\phi}_{j}(\mb{\xi}) + \mb{\phi}_{j}(\mb{\eta})) + \frac{1}{2} a(\mb{\phi}_{i}(\mb{\xi}) + \mb{\phi}_{i}(\mb{\eta}), \mb{\phi}_{j}(\mb{\xi}) - \mb{\phi}_{j}(\mb{\eta}))]^{2}.
    \end{align*}
    The inequality $(x + y)^{2} \leq 2 (x^{2} + y^{2})$ and the symmetry of $a$ provide
    \begin{align*}
        \sigma_{\max}(\mb{A}(\mb{\xi}) - \mb{A}(\mb{\eta}))^{2} & \leq \frac{1}{2} \sum_{i, j = 1}^{\dl} [a(\mb{\phi}_{i}(\mb{\xi}) - \mb{\phi}_{i}(\mb{\eta}), \mb{\phi}_{j}(\mb{\xi}) + \mb{\phi}_{j}(\mb{\eta}))^{2} + a(\mb{\phi}_{i}(\mb{\xi}) + \mb{\phi}_{i}(\mb{\eta}), \mb{\phi}_{j}(\mb{\xi}) - \mb{\phi}_{j}(\mb{\eta}))^{2}] \\
                                                                & = \sum_{i, j = 1}^{\dl} a(\mb{\phi}_{i}(\mb{\xi}) - \mb{\phi}_{i}(\mb{\eta}), \mb{\phi}_{j}(\mb{\xi}) + \mb{\phi}_{j}(\mb{\eta}))^{2}.
    \end{align*}
    Now using the boundedness of $a$ and the definition of the $\norm{\cdot}_{U, 2}$ norm, we obtain
    \begin{align*}
        \sigma_{\max}(\mb{A}(\mb{\xi}) - \mb{A}(\mb{\eta}))^{2} & \leq \sum_{i, j = 1}^{\dl} \norm{a}_{U \times U}^{2} \norm{\mb{\phi}_{i}(\mb{\xi}) - \mb{\phi}_{i}(\mb{\eta})}_{U}^{2} \norm{\mb{\phi}_{j}(\mb{\xi}) + \mb{\phi}_{j}(\mb{\eta})}_{U}^{2} \\
                                                                & = \norm{a}_{U \times U}^{2} \norm{\mb{\phi}(\mb{\xi}) - \mb{\phi}(\mb{\eta})}_{U, 2}^{2} \norm{\mb{\phi}(\mb{\xi}) + \mb{\phi}(\mb{\eta})}_{U, 2}^{2}.
    \end{align*}
    Combining this bound of $\sigma_{\max}(\mb{A}(\mb{\xi}) - \mb{A}(\mb{\eta}))$ and the upper bound of $\norm{\best{\mb{w}}(\mb{\xi}) - \best{\mb{w}}(\mb{\eta})}_{2}$ above shows the second inequality of the lemma.

    The last inequality of the lemma is obtained by combining \lem{regularity-energy} with the first two inequalities.
\end{proof}

\subsection{Proof of \texorpdfstring{\lem{convexity}}{Lemma~\ref{lem:convexity}}}
\label{app:convexity}

\begin{proof}
    Let $\mb{\xi} \in \md{X}$ and $\mb{v} \in \RR^{\dnl}$. Since $\mb{\phi}$ is twice differentiable in $U$, we infer that
    $$\nabla^{2} \red{\mc{K}}(\mb{\xi}) = \DW \DX \mc{K}(\mb{w}, \mb{\xi})|_{\mb{w} = \best{\mb{w}}(\mb{\xi})} + \DX \DX \mc{K}(\mb{w}, \mb{\xi})|_{\mb{w} = \best{\mb{w}}(\mb{\xi})},$$
    where we used the fact that $\DW \DX = \DW \DX$ and $\DW \mc{K}(\best{\mb{w}}(\mb{\xi}), \mb{\xi}) = 0$. Therefore, it is enough to look at the hessian of $\mc{K}$ with respect to $\md{X}$. Recalling the expression of the hessian of the linear and bilinear forms, we write
    \begin{align*}
        \DX^{2} \mc{K}(\mb{w}, \mb{\xi}) (\mb{v}, \mb{v}) & = a(\DX \mc{R}(\mb{w}, \mb{\xi}) \mb{v}, \DX \mc{R}(\mb{w}, \mb{\xi}) \mb{v})                                                              \\
                                                          & + a(\mc{R}(\mb{w}, \mb{\xi}), \DX^{2} \mc{R}(\mb{w}, \mb{\xi}) (\mb{v}, \mb{v})) - \ell(\DX^{2} \mc{R}(\mb{w}, \mb{\xi}) (\mb{v}, \mb{v})) \\
                                                          & = a(\DX \mc{R}(\mb{w}, \mb{\xi}) \mb{v}, \DX \mc{R}(\mb{w}, \mb{\xi}) \mb{v})                                                              \\
                                                          & + a(\mc{R}(\mb{w}, \mb{\xi}) - u^{\star}, \DX^{2} \mc{R}(\mb{w}, \mb{\xi}) (\mb{v}, \mb{v})),
    \end{align*}
    where we used the second-order conformity and the Galerkin orthogonality for the continuous solution. The Cauchy-Schwarz inequality then yields
    \begin{align*}
        \DX^{2} \mc{K}(\mb{w}, \mb{\xi}) (\mb{v}, \mb{v}) & \geq \norm{\DX \mc{R}(\mb{w}, \mb{\xi}) \mb{v}}_{a}^{2} - \norm{\mc{R}(\mb{w}, \mb{\xi}) - u^{\star}}_{a} \norm{\DX^{2} \mc{R}(\mb{w}, \mb{\xi}) (\mb{v}, \mb{v})}_{a}.
    \end{align*}
    Next, we decompose
    $$\norm{\mc{R}(\mb{w}, \mb{\xi}) - u^{\star}}_{a} \leq \norm{\mc{R}(\mb{w}, \mb{\xi}) - v}_{a} + \norm{v - u^{\star}}_{a}$$
    for all $v \in V$. Choosing $v = \red{\mc{R}}(\mb{\xi}^{\star})$ for some $\mb{\xi}^{\star} \in \md{X}^{\star}$, the second term is equal to the distance from $u^{\star}$ to $V$. Evaluating the expression above at $\mb{w} = \best{\mb{w}}(\mb{\xi})$, we conclude
    $$\nabla^{2} \red{\mc{K}}(\mb{\xi}) (\mb{v}, \mb{v}) \geq \norm{\nabla \red{\mc{R}}(\mb{\xi}) \mb{v}}_{a}^{2} - (\norm{\red{\mc{R}}(\mb{\xi}) - \red{\mc{R}}(\mb{\xi}^{\star})}_{a} + \inf_{v \in V} \norm{u^{\star} - v}_{a}) \norm{\nabla^{2} \red{\mc{R}}(\mb{\xi}) (\mb{v}, \mb{v})}_{a},$$
    for all $\mb{\xi}^{\star} \in \md{X}^{\star}$.
\end{proof}

%% file: article.bib
@misc{companion,
  title         = {Energy minimisation using overlapping tensor-product free-knot B-splines},
  author        = {Magueresse, Alexandre and Badia, Santiago},
  year          = {2025},
  eprint        = {2508.17705},
  archiveprefix = {arXiv}
}

@book{ekeland1999convex,
  title     = {Convex analysis and variational problems},
  author    = {Ekeland, Ivar and Témam, Roger},
  year      = {1999},
  publisher = {SIAM},
  doi       = {10.1137/1.9781611971088}
}

@article{bungartz2004sparse,
  title     = {Sparse grids},
  author    = {Bungartz, Hans-Joachim and Griebel, Michael},
  journal   = {Acta numerica},
  volume    = {13},
  pages     = {147--269},
  year      = {2004},
  publisher = {Cambridge University Press},
  doi       = {10.1017/S0962492904000182}
}

@book{saad2003iterative,
  title     = {Iterative methods for sparse linear systems},
  author    = {Saad, Yousef},
  year      = {2003},
  publisher = {SIAM},
  doi       = {10.1137/1.9780898718003}
}

@article{ainsworth1997posteriori,
  title     = {A posteriori error estimation in finite element analysis},
  author    = {Ainsworth, Mark and Oden, J Tinsley},
  journal   = {Computer methods in applied mechanics and engineering},
  volume    = {142},
  number    = {1-2},
  pages     = {1--88},
  year      = {1997},
  publisher = {Elsevier},
  doi       = {10.1016/S0045-7825(96)01107-3}
}

@article{babuvska1994p,
  title     = {The $p$- and $hp$- versions of the finite element method, basic principles and properties},
  author    = {Babu{\v{s}}ka, Ivo and Suri, Manil},
  journal   = {SIAM review},
  volume    = {36},
  number    = {4},
  pages     = {578--632},
  year      = {1994},
  publisher = {SIAM},
  doi       = {10.1137/1036141}
}

@article{budd2009adaptivity,
  title     = {Adaptivity with moving grids},
  author    = {Budd, Chris J and Huang, Weizhang and Russell, Robert D},
  journal   = {Acta Numerica},
  volume    = {18},
  pages     = {111--241},
  year      = {2009},
  publisher = {Cambridge University Press},
  doi       = {10.1017/S0962492906400015}
}

@book{huang2010adaptive,
  title     = {Adaptive moving mesh methods},
  author    = {Huang, Weizhang and Russell, Robert D},
  volume    = {174},
  year      = {2010},
  publisher = {Springer Science \& Business Media},
  doi       = {10.1007/978-1-4419-7916-2}
}

@article{lawton1971elimination,
  title     = {Elimination of linear parameters in nonlinear regression},
  author    = {Lawton, William H and Sylvestre, Edward A},
  journal   = {Technometrics},
  volume    = {13},
  number    = {3},
  pages     = {461--467},
  year      = {1971},
  publisher = {Taylor \& Francis},
  doi       = {10.2307/1267160}
}

@book{nocedal1999numerical,
  title     = {Numerical optimization},
  author    = {Nocedal, Jorge and Wright, Stephen J},
  year      = {1999},
  publisher = {Springer},
  doi       = {10.1007/978-0-387-40065-5}
}

@book{polyak1987introduction,
  title     = {Introduction to optimization},
  author    = {Polyak, Boris T},
  year      = {1987},
  publisher = {New York, Optimization Software}
}

@book{nesterov2013introductory,
  title     = {Introductory lectures on convex optimization: A basic course},
  author    = {Nesterov, Yurii},
  volume    = {87},
  year      = {2013},
  publisher = {Springer Science \& Business Media},
  doi       = {10.1007/978-1-4419-8853-9}
}

@article{bubeck2015convex,
  title     = {Convex optimization: Algorithms and complexity},
  author    = {Bubeck, S{\'e}bastien and others},
  journal   = {Foundations and Trends{\textregistered} in Machine Learning},
  volume    = {8},
  number    = {3-4},
  pages     = {231--357},
  year      = {2015},
  publisher = {Now Publishers, Inc.},
  doi       = {10.1561/2200000050}
}

@article{yashtini2016global,
  title     = {On the global convergence rate of the gradient descent method for functions with H{\"o}lder continuous gradients},
  author    = {Yashtini, Maryam},
  journal   = {Optimization letters},
  volume    = {10},
  pages     = {1361--1370},
  year      = {2016},
  publisher = {Springer},
  doi       = {10.1007/s11590-015-0936-x}
}

@article{nesterov2015universal,
  title     = {Universal gradient methods for convex optimization problems},
  author    = {Nesterov, Yurii},
  journal   = {Mathematical Programming},
  volume    = {152},
  number    = {1},
  pages     = {381--404},
  year      = {2015},
  publisher = {Springer},
  doi       = {10.1007/s10107-014-0790-0}
}

@misc{kingma2014adam,
  title         = {Adam: A method for stochastic optimization},
  author        = {Kingma, Diederik P and Ba, Jimmy},
  year          = {2014},
  eprint        = {1412.6980},
  archiveprefix = {arXiv}
}

@article{amari1998natural,
  title     = {Natural gradient works efficiently in learning},
  author    = {Amari, Shun-Ichi},
  journal   = {Neural computation},
  volume    = {10},
  number    = {2},
  pages     = {251--276},
  year      = {1998},
  publisher = {MIT Press},
  doi       = {10.1162/089976698300017746}
}

@article{martens2020new,
  title   = {New insights and perspectives on the natural gradient method},
  author  = {Martens, James},
  journal = {Journal of Machine Learning Research},
  volume  = {21},
  number  = {146},
  pages   = {1--76},
  year    = {2020}
}

@article{nurbekyan2023efficient,
  title     = {Efficient natural gradient descent methods for large-scale PDE-based optimization problems},
  author    = {Nurbekyan, Levon and Lei, Wanzhou and Yang, Yunan},
  journal   = {SIAM Journal on Scientific Computing},
  volume    = {45},
  number    = {4},
  pages     = {A1621--A1655},
  year      = {2023},
  publisher = {SIAM},
  doi       = {10.1137/22M147780}
}

@article{bioli2025accelerating,
  title         = {Accelerating Natural Gradient Descent for PINNs with Randomized Numerical Linear Algebra},
  author        = {Bioli, Ivan and Marcati, Carlo and Sangalli, Giancarlo},
  year          = {2025},
  eprint        = {2505.11638},
  archiveprefix = {arXiv}
}

@inproceedings{jin2017escape,
  title        = {How to escape saddle points efficiently},
  author       = {Jin, Chi and Ge, Rong and Netrapalli, Praneeth and Kakade, Sham M and Jordan, Michael I},
  booktitle    = {International conference on machine learning},
  pages        = {1724--1732},
  year         = {2017},
  organization = {PMLR}
}

@inproceedings{karimi2016linear,
  title        = {Linear convergence of gradient and proximal-gradient methods under the Polyak-{\L}ojasiewicz condition},
  author       = {Karimi, Hamed and Nutini, Julie and Schmidt, Mark},
  booktitle    = {Joint European conference on machine learning and knowledge discovery in databases},
  pages        = {795--811},
  year         = {2016},
  organization = {Springer},
  doi          = {10.1007/978-3-319-46128-1_50}
}

@article{traonmilin2023basins,
  title     = {The basins of attraction of the global minimizers of non-convex inverse problems with low-dimensional models in infinite dimension},
  author    = {Traonmilin, Yann and Aujol, Jean-Fran{\c{c}}ois and Leclaire, Arthur},
  journal   = {Information and Inference: A Journal of the IMA},
  volume    = {12},
  number    = {1},
  pages     = {113--156},
  year      = {2023},
  publisher = {Oxford University Press},
  doi       = {10.1093/imaiai/iaac011}
}

@article{raissi2019physics,
  title     = {Physics-informed neural networks: A deep learning framework for solving forward and inverse problems involving nonlinear partial differential equations},
  author    = {Raissi, Maziar and Perdikaris, Paris and Karniadakis, George E},
  journal   = {Journal of Computational physics},
  volume    = {378},
  pages     = {686--707},
  year      = {2019},
  publisher = {Elsevier},
  doi       = {10.1016/j.jcp.2018.10.045}
}

@article{weinan2018deep,
  title     = {The Deep Ritz Method: A Deep Learning-Based Numerical Algorithm for Solving Variational Problems},
  volume    = {6},
  issn      = {2194-671X},
  doi       = {10.1007/s40304-018-0127-z},
  number    = {1},
  journal   = {Communications in Mathematics and Statistics},
  publisher = {Springer Science and Business Media LLC},
  author    = {E,  Weinan and Yu,  Bing},
  year      = {2018}
}

@article{beck2022full,
  title     = {Full error analysis for the training of deep neural networks},
  author    = {Beck, Christian and Jentzen, Arnulf and Kuckuck, Benno},
  journal   = {Infinite Dimensional Analysis, Quantum Probability and Related Topics},
  volume    = {25},
  number    = {02},
  pages     = {2150020},
  year      = {2022},
  publisher = {World Scientific},
  doi       = {10.1142/S021902572150020X}
}

@article{mishra2023estimates,
  title     = {Estimates on the generalization error of physics-informed neural networks for approximating PDEs},
  author    = {Mishra, Siddhartha and Molinaro, Roberto},
  journal   = {IMA Journal of Numerical Analysis},
  volume    = {43},
  number    = {1},
  pages     = {1--43},
  year      = {2023},
  publisher = {Oxford University Press},
  doi       = {10.1093/imanum/drab032}
}

@article{de2024numerical,
  title     = {Numerical analysis of physics-informed neural networks and related models in physics-informed machine learning},
  author    = {De Ryck, Tim and Mishra, Siddhartha},
  journal   = {Acta Numerica},
  volume    = {33},
  pages     = {633--713},
  year      = {2024},
  publisher = {Cambridge University Press},
  doi       = {10.1017/S0962492923000089}
}

@article{petersen2021topological,
  title     = {Topological properties of the set of functions generated by neural networks of fixed size},
  author    = {Petersen, Philipp and Raslan, Mones and Voigtlaender, Felix},
  journal   = {Foundations of computational mathematics},
  volume    = {21},
  pages     = {375--444},
  year      = {2021},
  publisher = {Springer},
  doi       = {10.1007/s10208-020-09461-0}
}

@article{jupp1978approximation,
  title     = {Approximation to data by splines with free knots},
  author    = {Jupp, David LB},
  journal   = {SIAM Journal on Numerical Analysis},
  volume    = {15},
  number    = {2},
  pages     = {328--343},
  year      = {1978},
  publisher = {SIAM},
  doi       = {10.1137/0715022}
}

@article{schutze2003bivariate,
  title     = {Bivariate free knot splines},
  author    = {Sch{\"u}tze, Torsten and Schwetlick, Hubert},
  journal   = {BIT Numerical Mathematics},
  volume    = {43},
  pages     = {153--178},
  year      = {2003},
  publisher = {Springer},
  doi       = {10.1023/A:1023609324173}
}

@article{beliakov2004least,
  title     = {Least squares splines with free knots: global optimization approach},
  author    = {Beliakov, Gleb},
  journal   = {Applied mathematics and computation},
  volume    = {149},
  number    = {3},
  pages     = {783--798},
  year      = {2004},
  publisher = {Elsevier},
  doi       = {10.1016/S0096-3003(03)00179-6}
}

@inproceedings{deng2004optimizing,
  title        = {On optimizing knot positions for multidimensional B-spline models},
  author       = {Deng, Xiang and Denney Jr, Thomas S},
  booktitle    = {Computational Imaging II},
  volume       = {5299},
  pages        = {175--186},
  year         = {2004},
  organization = {SPIE},
  doi          = {10.1117/12.527245}
}

@article{zhang2016b,
  title     = {B-spline surface fitting with knot position optimization},
  author    = {Zhang, Yuhua and Cao, Juan and Chen, Zhonggui and Li, Xin and Zeng, Xiao-Ming},
  journal   = {Computers \& Graphics},
  volume    = {58},
  pages     = {73--83},
  year      = {2016},
  publisher = {Elsevier},
  doi       = {10.1016/j.cag.2016.05.010}
}

@article{kovacs2019nonlinear,
  title     = {Nonlinear least-squares spline fitting with variable knots},
  author    = {Kov{\'a}cs, P{\'e}ter and Fekete, Andrea M},
  journal   = {Applied Mathematics and Computation},
  volume    = {354},
  pages     = {490--501},
  year      = {2019},
  publisher = {Elsevier},
  doi       = {10.1016/j.amc.2019.02.051}
}

@article{he2020relu,
  title     = {Relu deep neural networks and linear finite elements},
  author    = {He, Juncai and Li, Lin and Xu, Jinchao and Zheng, Chunyue},
  journal   = {Journal of Computational Mathematics},
  volume    = {38},
  number    = {3},
  pages     = {502--527},
  year      = {2020},
  publisher = {Inst. of Computational Mathematics and Sc./Eng. Computing},
  doi       = {10.4208/jcm.1901-m2018-0160}
}

@article{opschoor2020deep,
  title     = {Deep ReLU networks and high-order finite element methods},
  author    = {Opschoor, Joost AA and Petersen, Philipp C and Schwab, Christoph},
  journal   = {Analysis and Applications},
  volume    = {18},
  number    = {05},
  pages     = {715--770},
  year      = {2020},
  publisher = {World Scientific},
  doi       = {10.1142/S0219530519410136}
}

@article{daubechies2022nonlinear,
  title     = {Nonlinear approximation and (deep) ReLU networks},
  author    = {Daubechies, Ingrid and DeVore, Ronald and Foucart, Simon and Hanin, Boris and Petrova, Guergana},
  journal   = {Constructive Approximation},
  volume    = {55},
  number    = {1},
  pages     = {127--172},
  year      = {2022},
  publisher = {Springer},
  doi       = {10.1007/s00365-021-09548-z}
}

@article{yarotsky2017error,
  title     = {Error bounds for approximations with deep ReLU networks},
  author    = {Yarotsky, Dmitry},
  journal   = {Neural networks},
  volume    = {94},
  pages     = {103--114},
  year      = {2017},
  publisher = {Elsevier},
  doi       = {10.1016/j.neunet.2017.07.002}
}

@article{petersen2018optimal,
  title     = {Optimal approximation of piecewise smooth functions using deep ReLU neural networks},
  author    = {Petersen, Philipp and Voigtlaender, Felix},
  journal   = {Neural Networks},
  volume    = {108},
  pages     = {296--330},
  year      = {2018},
  publisher = {Elsevier},
  doi       = {10.1016/j.neunet.2018.08.019}
}
